\documentclass{amsart}

\usepackage{amsmath}
\usepackage{amssymb}
\usepackage{amsfonts}
\usepackage{graphicx}
\usepackage[plainpages=false,pdfpagelabels, backref=page]{hyperref}
\usepackage[nameinlink]{cleveref}
\usepackage{slashed}
\usepackage[shortlabels]{enumitem}
\usepackage{tikz-cd}

\hypersetup{%
    colorlinks=true,
    linkcolor=cyan,
    citecolor=magenta, 
    pdfborderstyle={/S/U/W 1}
}

\setlist[enumerate]{%
    labelsep=8pt,%
    labelindent=0.5\parindent,%
    itemindent=0pt,%
    leftmargin=*,%
    listparindent=-\leftmargin,%
    itemsep = 0pt
}

\theoremstyle{plain}
\newtheorem{theorem}{Theorem}[section]
\newtheorem{lemma}[theorem]{Lemma}
\newtheorem{corollary}[theorem]{Corollary}
\newtheorem{proposition}[theorem]{Proposition}
\newtheorem{mythm}{Theorem}

\newtheorem{myprop}[mythm]{Proposition}

\theoremstyle{definition}
\newtheorem{definition}[theorem]{Definition}

\theoremstyle{remark}
\newtheorem{remark}[theorem]{Remark}

\renewcommand{\hat}{\widehat}
\renewcommand{\tilde}{\widetilde}

\newcommand{\real}{\mathbb{R}}
\newcommand{\complex}{\mathbb{C}}

\newcommand{\mass}{\mathop{}\!\mathrm{d}}
\newcommand{\slcover}{\tilde{SL_2(\real)}}
\newcommand{\bog}{\bar{\Omega}_G}
\newcommand{\wolfclass}{\tilde{\mathcal{H}}}

\DeclareMathOperator{\cadj}{coad}

\DeclareMathOperator{\Hom}{Hom}
\DeclareMathOperator{\Ad}{Ad}
\DeclareMathOperator{\ad}{ad}
\DeclareMathOperator{\ind}{ind}
\DeclareMathOperator{\End}{End}

\title[Schwartz estimate on Lie groups]{On the Schwartz Estimate for Hodge Laplacians on semisimple Lie Groups}
\author{Zhicheng Han}
\address{George-August-Universit\"at G\"ottingen, Mathematisches Institut, Bunsenstra\ss e 3-5, 37073 G\"ottingen, Germany}
\email{zhicheng.han@mathematik.uni-gottingen.de}

\begin{document}

\begin{abstract}
    In this paper, we prove Schwartz estimates for Hodge Laplacian and Dirac operators on semisimple Lie groups. Alongside, we gives a version of Kuga's lemma for its Lie algebra cohomology. This is a generalization of similar results of Barbasch and Moscovici \cite{Barbasch1983} on symmetric spaces. The main purpose of such estimates is to study the heat problem not only in the scalar case, but also for sections of vector bundles on homogeneous spaces using Fourier analysis.
\end{abstract}

\maketitle

\tableofcontents

\section{Introduction}
Let $G$ be a real, connected Lie group with infinite center. Denote its Lie algebra as $\mathfrak{g}$. Furnish $G$ with a left-$G$-invariant Riemannian metric. Let $\Lambda^*G$ be the bundle of differential forms over $G$, and let $\Delta_p$ be the Hodge Laplacian acting on the $p$-forms of $G$. This operators is well-known to be a second-order elliptic operator on $G$, which admits a kernel of the heat exonerator $e^{-t\Delta_p}$, which we denote as $k_t^p(g)$.

It is our intention to study the spectral properties of $\Delta_p$ on a semisimple Lie group with infinite center as group manifold. We are particularly interested in the spectral properties of $\Delta_p$ on the $L^2$-spaces. The spectral properties of $\Delta_p$ are closely related to the representation theory of $G$.

The main purpose of this paper is to prove that given the most invariant metric on $G$, we can prove that the heat kernel $k^p_t(g)$ associated with the corresponding Hodge Laplacian $\Delta_p$ is a Schwartz function, defined in the sense of Herb and Wolf \cite{HerbWolf86II}. 

To be more precise, let $G$ be a connected reductive Lie group of class $\wolfclass$. This is a class defined by Wolf that contains all semisimple Lie groups with infinite center \cite{Wolf1974}. If we fix $\theta$ a Cartan involution of $G$, and denotes its fixed point to be $K$, then there is a canonical way to construct a left $G$-invariant metric on $G$ that is also $K$-bi-invariant from the Killing form $B$ of $G$. Denote this metric to be $B^\theta$. Our main result is then stated as the following:
\begin{mythm}[\textbf{Schwartz kernel}]\label{thm: Schwartz kernel}
    Let $t>0$. Then for any connected reductive Lie group $G$ of class $\wolfclass$ and corresponding metric $B^\theta$, the kernels of the Hodge Laplacian $k^p_t$ associated with the differential forms of $G$ are $L^q$-Schwartz functions on $G$, in the sense of Harish-Chandra, for all $q>0$.
\end{mythm}
The property of the function to be of Schwartz class is a vital condition for many applications in sight, because it is the class of functions where the most general version of Plancherel formula is established in current literature. We will later deploy the Plancherel formula in another paper \cite[TOCHANGE]{Thesis}, to study the spectral properties of $\Delta_p$ on groups such as $\slcover$. The proof that $k_t^p$ is a Schwartz function is a crucial step in such applications.

Such a class of Schwartz function denoted by $\mathcal{S}(G)$ was first defined by Harish-Chandra \cite[Section~6]{harish1966} on reductive Lie groups with compact center and later extended by Herb and Wolf \cite{HerbWolf86II} to the wider class of semisimple groups with infinite center. The latter class of Schwartz functions was suitably defined to capture the growth of functions in the $K$-direction, which is now noncompact.

The most general Plancherel formula known to the author was established by Herb and Wolf in the same paper, extending the previous work of Harish-Chandra \cite{harish1975, Harish-Chandra1976II, Harish-chandra1976III}. The main result of this paper therefore validates the application of the Plancherel formula and subsequently representation theory of $G$ to studying the spectral problems of $\Delta_p$ on group manifolds $G$.

The left $G$-invariant differential operators on $G$ can be identified with elements of the universal enveloping algebra $U( \mathfrak{g}_\complex )$. Following the same methodology, we express the Hodge Laplacians in terms of actions of $U( \mathfrak{g}_\complex )$. This is a common strategy used in the study of symmetric spaces, where the Hodge Laplacian is essentially the Casimir element $\Omega_G$ in the center of $U( \mathfrak{g}_\complex )$.

We begin with an expression for the Hodge Laplacian associated with any Lie group $G$ and any left $G$-invariant metric. In the case $G$ is reductive, we fix the most invariant Riemannian metric in sight, that is the Killing form $B$ twisted with the Cartan involution $\theta$:
\begin{equation}
    B^\theta(X, Y) = B(X, \theta Y)
\end{equation}
for $X, Y\in  \mathfrak{g}_\complex $. This exploits the underlying symmetry and gives a refined version of Kuga's lemma on the Lie algebra cohomology of $ \mathfrak{g}$ twisted with any unitary $ \mathfrak{g}$-module $V$ (see \Cref{thm: rep_Laplacian} and \Cref{thm: generalized_Kuga}):
\begin{myprop}
    Given a unitary $\mathfrak{g}$-module $(\pi, V_\pi)$, the Hodge Laplacian $\Delta_p$ on $C^\infty(\mathfrak{g}, V_\pi)$ can be expressed explicitly. If $G$ is further assumed to be reductive, then this expression for the Hodge Laplacian associated with $B^\theta$ is given in terms of the representation $\pi$ and the adjoint representation $\ad$ of $ \mathfrak{g} $.
\end{myprop}

This is also where we depart from the case of symmetric spaces. In both cases, the differential operator $D$ in question can be written as
\begin{equation*}
    D = -R(\sum_i X_i^2) + B
\end{equation*}
as the sum of the Laplacian $-\sum_i X_i^2$ on functions of $G$ and a perturbation term $B$. In the case of symmetric spaces, $B$ is essentially a scalar operator, whereas in our cases the perturbation term is typically a first-order differential operator with variable coefficients. Moreover, the fact that $K$ is noncompact raises yet further difficulties when we estimate derivatives from both sides.

The main novelty in this paper is first to give a very explicit description of the Hodge Laplacian, so that the perturbation term and corresponding perturbation estimates can be made. Secondly, we incorporate the perturbation theory into the Gaussian estimates of the heat kernel. This allows us to estimate the kernel under derivatives from one side for a large class of differential operators on homogeneous vector bundles, which in particular contains the Hodge Laplacian on Lie groups. Lastly, we also extends an argument of Harish-Chandra to estimate the kernel when $K$ is noncompact.

The paper is organized as follows. In \Cref{section: preliminaries} we collect some general facts about groups of class $\wolfclass$. We also introduce the notion of Schwartz functions on such groups. In \Cref{section: Lie algebra cohomology} we give an explicit expression of the Hodge Laplacian on general Lie algebra cohomology $C^*( \mathfrak{g}; V )$. This allows us to distinguish the first-order perturbation operator. In \Cref{section: elliptic operators} we collect some results on perturbation theory for the heat semigroup. This allows us to control the kernel estimates. Eventually in \Cref{section: proof of the Schwartz estimate} we assemble all the ingredients to prove the Schwartz estimate for the Hodge Laplacian. Lastly in \Cref{section: nilpotent} we discuss how this method can be readily applied to prove a similar result in nilpotent Lie groups. This is an easy consequence derived from the estimates obtained in \Cref{section: proof of the Schwartz estimate} that will be useful in a subsequent paper.

In a subsequent paper \cite[TOCHANGE]{Thesis} we will use the Schwartz estimate to study the spectral properties of the Hodge Laplacian on semisimple Lie groups such as $\slcover$. A similar method is also used to study the spectral properties of the Dirac operator on such groups.


This paper is a part of the author's PhD thesis at the University of G\"ottingen \cite{Thesis}. The author would like to thank his advisor Thomas Schick for suggesting the research question and for carefully reading the first draft.

\section{Preliminaries} \label{section: preliminaries}
In this section we collect some facts about the structure of general semisimple Lie groups with infinite center. We also recall the notion of Schwartz functions on such groups, the Plancherel formula of which was established by Herb and Wolf \cite{HerbWolf86II}. 

Let $G$ be a reductive Lie group. This implies that its Lie algebra admits the following decomposition:
\begin{equation*}
    \mathfrak{g} =  Z_\mathfrak{g}\oplus\mathfrak{g}_0 \quad \text{where $ Z_\mathfrak{g}$ is central and  $\mathfrak{g}_0 =  [\mathfrak{g},  \mathfrak{g} ]$ semisimple  }
\end{equation*}
In \cite{Wolf1974}, Wolf defined a reductive Lie group $G$ to be of class $\wolfclass$ if it satisfies the following properties:
\begin{enumerate}
    \item If $g\in G$ then $\ad(g)$ is an inner automorphism of $ \mathfrak{g}_\complex$;
    \item It contains a closed normal abelian subgroup $Z$ that centralizes the identity component $G^0$ of $G$; 
    \item $ZG^0$ has finite index in $G$;
    \item $Z\cap G^0$ is cocompact in the center $Z_{G^0}$ of $G^0$.
\end{enumerate}
As discussed in \cite[\S 0]{Wolf1974}, this is a convenient class of reductive groups that contains every connected semisimple group and is stable under passage to Levi components of cuspidal parabolic subgroups of $G$. In particular this contains the original class of Harish-Chandra, which are groups where $ G/G^0$ and the center of $[G^0, G^0]$ are compact. In addition it also contains the class of connected semisimple groups with infinite center, such as the universal covers of linear reductive groups like $SL(2, \real)$, $Sp(n, \real)$ and $U(n,n)$.

We assume throughout this paper that $G$ is connected. This lightens some of the exposition significantly. Note in such case $ZG^0 = G^0 = G$.

A \textit{Cartan involution} on $G$ is an involutive automorphism $\Theta$ such that the fixed point set $K = G^\Theta$, which we call the \textit{relative maximal compact subgroup}, is the full inverse image of a maximal compact subgroup of $\Ad(G)$. Note that $K$ is only compact if the center $Z_G$ of $G$ is compact. Fix a Cartan involution $\Theta$ of $G$ and let $K$ be the corresponding relative maximal compact subgroup. Then $K$ has a unique maximal compact subgroup $K_1$ and has a closed normal vector subgroup $V_G$ such that \cite[Proposition~2.1]{HerbSchwartz1}:
\begin{enumerate}
    \item $K = K_1\times V_G$;
    \item $Z = Z_G\cap V_G$ is cocompact in both $V_G$ and $Z_G$.
\end{enumerate}
The Cartan involution $\Theta$ differentiates to give a Cartan involution $\theta$ on $ \mathfrak{g}$.  Denote $ \mathfrak{k, p}$ to be the eigenspaces of $\theta$ corresponding to the eigenvalues $ + 1$ and $-1$ respectively. Choose a maximal abelian subspace $ \mathfrak{a_p} $ of $ \mathfrak{p}$ and form the restricted root system $\Delta^+( \mathfrak{g};  \mathfrak{a_p} )$. By picking the polarity $\Delta^+( \mathfrak{g};  \mathfrak{a_p} )$ on $\Delta( \mathfrak{g};  \mathfrak{a_p}  )$, we form an Iwasawa decomposition
\begin{equation}\label{exp: Iwasawa}
     \mathfrak{g} =  \mathfrak{k}  \oplus  \mathfrak{a_p} \oplus  \mathfrak{n_p}
\end{equation} 
where $ \mathfrak{n_p}$ is the sum of positive $ \mathfrak{a_p}$-root spaces. Denote $A$ and $N$ to be the exponential of $ \mathfrak{a_p}$ and $ \mathfrak{n_p}$ respectively. At group level this gives a Iwasawa decomposition $G = NAK$
\begin{equation}
    x = n(x)\cdot \exp H(x) \cdot \kappa(x),
\end{equation} 
where $n(x) \in  N$, $H(x)\in  \mathfrak{a} $ and $\kappa(x) \in K$. Denote 
\begin{equation*}
    \rho_{ \mathfrak{p} } = \frac{1}{2}\sum_{\alpha\in \Delta^+( \mathfrak{g};  \mathfrak{a_p}  )}\alpha
\end{equation*}
the half sum of positive restricted roots, counted with multiplicity. Define now the zonal spherical function in a similar way as in the classical case (c.f. \cite[\S VII.8]{knapp2016}):
\begin{equation}
    \phi_0^{G/Z}(g):= \int_{K/Z} e^{-(\rho_\mathfrak{p})H(kg)}\mass{(kZ)}.
\end{equation}
Note this enjoys all the desired properties of the classical zonal spherical functions: inverse symmetries, bi-$K$-invariance, and of moderate growth (c.f. \cite[(2.5)]{HerbWolf86II}), i.e.:
\begin{enumerate}[wide = 0pt, label=(\roman*)]
    \item For all $k_1, k_2\in K, g\in G$, $\phi_0^{G/Z}(k_1gk_2) = \phi_0^{G/Z}(g)$ and $\phi_0^{G/Z}(g) = \phi_0^{G/Z}(g^{-1})$;
    \item For some $d\geq 0$ and for all $a\in \exp(\overline{ \mathfrak{a}^+})$:
    \begin{equation}\label{exp: spherical vector estimate}
        \phi_0^{G/Z}(a) \leq C a^{-\rho_ \mathfrak{p}} (1+ \rho_ \mathfrak{p}\log a )^d
    \end{equation}
\end{enumerate}
In particular, $\phi^{G/Z}_0$ is tempered: $\phi_0^{G/Z}\in L^{2+\epsilon}(G/Z)$ for any $\epsilon>0$. 

We also extend the original norm on $ \mathfrak{p}$ suitably to capture the growth in the vector subgroup $V_G$ of $K$. Recall that the Killing form $B(-, -)$ is positive definite on $ \mathfrak{p}$ and negative definite on $ \mathfrak{k}$, and is $G$-bi-invariant. Given 
\begin{equation} \label{exp: refined Cartan decomposition}
    x = v(x)\cdot \kappa_1(x)\cdot \exp X 
\end{equation}
with respect to the Cartan decomposition $G = V_GK_1\exp \mathfrak{p}$, we define
\begin{equation}
    |x|_ \mathfrak{ \mathfrak{p} } :=  B(X,X)^{1/2}, \quad |x|_v = -B(v(x), v(x))^{1/2}.
\end{equation}
This define two seminorms on $G$. Define now $|\cdot|_{ \mathfrak{pz}} :G\to \real_{\geq 0}$ by:
\begin{equation}
    |x|_{ \mathfrak{pz}} := |x|_v + | x |_ \mathfrak{p}
\end{equation}
Note that $|x|_{ \mathfrak{pz}}$ grows at the same rate as $|x|_\mathfrak{p}$ if $Z_G$ is compact. Moreover, $|\cdot|_{\mathfrak{pz}}$ is $\Ad(K)$-invariant, $K_1$-bi-invariant \cite[(2.9)]{HerbSchwartz1}. 

We now define the Schwartz functions on $G$ in a similar way as Harish-Chandra, replacing $|\cdot|_ \mathfrak{p}$ by $|\cdot|_{ \mathfrak{pz}}$. Denote $U( \mathfrak{g}_\complex )$ as the universal enveloping algebra of $ \mathfrak{g}$. It acts on functions of $G$ by differentiation from both left and right. We denote the corresponding action as $L(D)$ and $R(D)$ for $D\in U( \mathfrak{g}_\complex )$.

Given $D_1, D_2\in  U( \mathfrak{g}_\complex ) $ and $r\in \real$, we define a family of seminorms:
    \begin{equation} \label{exp: Schwartz seminorm}
        v^p_{D_1, D_2, r}(f) = \sup_{x\in G}\bigg|(1+ \left\lvert x \right\rvert _\mathfrak{pz} )^r\phi_0^{G/Z}(x)^{-2/p}L(D_1)R(D_2)F(x)\bigg|
    \end{equation}
for $f\in C^\infty(G)$. Then we define \emph{Harish-Chandra $L^p$-Schwartz space} on $G$ to be
\begin{equation} \label{exp: Schwartz space}
    \mathcal{S}^p(G) = \left\{ f\in C^\infty(G): v^p_{D_1, D_2, r}(F)< \infty \text{ for all }D_1,D_2 \in U( \mathfrak{g}_\complex ), r\in \real\right\}.
\end{equation}
We will prove later the kernel functions we care about lie in $S^p(G)$ for all $p>0$. 

As in the classical case, for all$f\in \mathcal{S}^2(G)$, the Plancherel formula of groups of class $\wolfclass$ is established \cite[Theorem~7.6]{HerbWolf86II}. Moreover, $\mathcal{S}^2(G)$ a dense subspace of $L^2(G)$, and is a topological algebra under convolution \cite[\S 6]{HerbWolf86II}.

\section{Lie algebra cohomology} \label{section: Lie algebra cohomology}

In this section we study the Lie algebra cohomology of $\mathfrak{g}$ on arbitrary connected Lie group $ \mathfrak{g}$. We first express the Hodge Laplacian on the Lie algebra cohomology in a general form, and then we specialize it to the case of $\mathfrak{g}$ is reductive. We will see how the terms simplified significantly in this case. 

First we recall some basics of Lie algebra cohomology. The details can be found in \cite{borel2013} In particular, fix a $ \mathfrak{g}$-module $(V,\tau)$ over the field $F = \real$ or $\complex$, where $\tau$ denotes the $ \mathfrak{g}$-module structure of $V$. The $C^q=C^q(\mathfrak{g}; V)=\Hom_F(\Lambda^q\mathfrak{g}, V)$ together with differentials $d:C^q\to C^{q+1}$
\begin{equation}\label{thm: borel1.1.1}
    \begin{split}
        df(X_0,\cdots, X_q)=\sum_i&(-1)^i\tau(X_i) f(X_0, \cdots, \hat{X_i}, \cdots, X_q)\\
    &+\sum_{i<j}(-1)^{i+j} f([X_i, X_j], X_0, \cdots , \hat{X_i}, \cdots, \hat{X_j}, \cdots, X_q).
    \end{split}
\end{equation}
As usual, the $\hat{\cdot}$ stands for omission of the argument. Then we denote $H^*( \mathfrak{g}; V)$ the cohomology of the complex. We denote the first sum as $d_\circ$ and the second sum as $d_{\wedge}$ respectively, that is: $d = d_\circ + d_{\wedge}$. 

Fix now a left $G$-invariant Riemannian metric on $G$. Denote $\cadj:  \mathfrak{g}\to \End(\mathfrak{g}^*_\complex)$ the coadjoint representation of $G$ on $\wedge^*\mathfrak{g}^*$, defined by
\begin{equation}\label{exp: coadj_rep}
    \cadj X(l)(Y) := l([Y, X]) = l(\ad_{-X}Y)
\end{equation}
for $ X, Y\in  \mathfrak{g}$ and $l\in  \mathfrak{g}^*_\complex$. The Riemannian metric also induces a natural hermitian inner product on $\bigwedge^* \mathfrak{g}^*_\complex$. We further assume the $ \mathfrak{g}$-module $V$ is hermitian with hermitian product $\langle -,- \rangle _V$, i.e.:
\begin{equation}
    \langle \tau(X)u, v \rangle_V + \langle u, \tau(X)v \rangle_V = 0
\end{equation} 
for all $u, v\in V$ and all $X\in \mathfrak{g}$. Then $C^*( \mathfrak{g}; V) \cong V\otimes \wedge^* \mathfrak{g}^*_\complex$ is furnished with a hermitian structure, which we denote as $\langle -, - \rangle$.

Follow the convention of \cite[\S II.1.4]{borel2013}, we define the adjoint operator $\delta$ to $d$ with respect to $\langle -,- \rangle$: Fix an orthonormal basis $X_1, \dots, X_n$, Denote $\omega^i$ the dual basis of $ \mathfrak{g}^*_\complex$ with respect to the perfect pairing $\omega_i(X_j) = \delta_{ij}$, and put
\begin{equation*}
    \omega^J = \omega^{j_1}\wedge \omega^{j_2} \wedge \cdots \wedge \omega^{j_q}
\end{equation*} 
for a multi-index $J = \{j_1, \cdots, j_q \}$ and $J_m = \{1, \cdots, m\}$.

Denote the structural constants associated with this Lie algebra as $C^\gamma_{\alpha, \beta}$, i.e.:
\begin{equation}\label{exp: const_struct}
    [X_\alpha, X_\beta] = \sum_{\gamma} C_{\alpha, \beta}^\gamma X_\gamma.
\end{equation}
It is immediate from \eqref{exp: coadj_rep} that:
\begin{equation}\label{exp: const_struct2}
    \cadj X_\alpha (\omega^\beta) = \sum_{\gamma} C_{\gamma, \alpha}^\beta \omega^\gamma.
\end{equation}
We denote the index set of $\mathfrak{g}$ as $I_\mathfrak{g}$. Whenever the index is unspecified, we meant to sum over the whole basis of $ \mathfrak{g}$.

If $\eta\in C^q( \mathfrak{g}; V )$ one writes $\eta_J = \eta(X_{j_1}, \cdots, X_{j_q})$ and $\eta = \sum_{J} \eta_J \cdot \omega^J$. In this way \eqref{thm: borel1.1.1} can be rewritten for $\eta\in C^q( \mathfrak{g}; V)$ as
\begin{equation*}
    (d_\circ\eta)_J = \sum_{1\leq u \leq {q+1}}(-1)^{u-1} \tau(X_{j_u})\cdot \eta_{J(u)} \qquad \text{for } J\subseteq J_m, |J| = {q+1}
\end{equation*}
where $J(u_1, u_2, \dots, u_n)$ denotes the $J$ with the $u_i$th entries removed for $i= 1, \dots, n$. Meanwhile,
\begin{equation} \label{exp: d_wedge}
    \begin{split}
        (d_\wedge \eta)_J &= \sum_{1\leq \alpha<\beta\leq {q+1}}(-1)^{\alpha+\beta}\eta_{[j_\alpha, j_\beta]\cup J(\alpha, \beta)}\\
        &= \sum_{1\leq \alpha<\beta\leq {q+1}}\sum_{j} (-1)^{\alpha+\beta} C_{j_\alpha, j_\beta}^j \eta_{j\cup J(\alpha, \beta)}
    \end{split}
\end{equation}
\begin{definition}
    Let $\delta: C^q( \mathfrak{g}; V) \to C^{q-1}( \mathfrak{g}; V )$ be the linear operator adjoint to $d$:
    \begin{equation*}
        \langle \delta \eta, \mu \rangle = \langle \eta, d\mu \rangle \quad \textrm{for all }\eta\in C^q( \mathfrak{g};  V), \mu\in C^{q-1}( \mathfrak{g};  V)
    \end{equation*}
\end{definition}
The summands $d_\circ$ and $d_\wedge$ admit formal adjoints $\delta_\circ $ and $\delta_\wedge$ respectively. 
Their action can be explicitly expressed on the basis as follows:
\begin{proposition} \label{thm: d_adj}
    The operator $\delta_\circ$ satisfies
    \begin{equation}
        (\delta_\circ \eta)_J = \sum_{j} \tau(X_j)^* \eta_{\{j\}\cup J},
    \end{equation}
    where $\tau(X)^*$ is the adjoint of $\tau(X)$ with respect to $\langle -, - \rangle_V$. The operator $\delta_\wedge$ satisfies
    \begin{equation}
        (\delta_\wedge \eta)_J = \sum_{\alpha<\beta}\sum_{1\leq u\leq q-1}(-1)^{u-1}C_{\alpha, \beta}^{j_u}\eta_{\{\alpha, \beta\}\cup J(u)}
    \end{equation}
    for $\eta\in C^q( \mathfrak{g}; V )$ and $|J| = q-1$.
\end{proposition}

\begin{proof}
    The first statement is a direct adaptation of the argument in \cite[Proposition~II.2.3]{borel2013}. As for the second, it suffices to verify the identity $\langle \delta_\wedge \eta, \nu \rangle = \langle \eta, d_\wedge \nu \rangle$ for $\eta\in C^q( \mathfrak{g}; V )$ and $\nu \in C^{q-1}( \mathfrak{g}; V )$. Note:
    \begin{equation}
        \langle \eta, d_\wedge \nu \rangle = \sum_{|J|=q}\langle \eta_J, (d_\wedge \nu)_J \rangle = \sum_{|J|=q}\langle \eta_J, \sum_{1\leq \alpha<\beta\leq {q}}\sum_{j} (-1)^{\alpha+\beta} C_{j_\alpha, j_\beta}^j \nu_{j\cup J(\alpha, \beta)} \rangle,
    \end{equation}
    and similarly 
    \begin{equation}    
            \langle \delta_\wedge \eta, \nu \rangle = \sum_{|J'| = q-1}\langle  \sum_{\alpha< \beta}\sum_{1\leq u\leq q-1} (-1)^{u-1}C_{\alpha, \beta}^{j'_u}\eta_{\{\alpha, \beta\}\cup J'(u)}, \nu_{J'} \rangle.
    \end{equation}
    For fixed index set $J$ of $\eta$ in the second expression, the indices $J'$ of $\nu$ associated with $\delta_\wedge\eta$ are
    \begin{equation*}
        \bigcup_{\alpha, \beta}\left\{ J'\in I_\mathfrak{g}\: \middle| \: |J'| =q-1 \quad \{\alpha, \beta\}\cup J'(u) = J \quad \text{for some } 1\leq u\leq q-1 \right\},
    \end{equation*}
    whereas for the same $J$, the corresponding indices of $J'$ associated with $(\delta_\wedge)_{J'}$ are
    \begin{equation*}
        \begin{split}
            &\bigcup_{1\leq u\leq q-1}\bigcup_{\alpha, \beta}\bigcup_{1\leq \alpha' < \beta' \leq q}\left\{ J' \: \middle| \:  j'_{\alpha'} = \alpha, j'_{\beta'} = \beta, \{\alpha, \beta\}\cup J'(u) = J \right\}\\
            = &\bigcup_{1\leq u\leq q-1}\bigcup_{1\leq \alpha' < \beta' \leq q}\left\{ J' \: \middle| \: J'(u) = J(\alpha', \beta') \right\} \\
            = & \bigcup_{j}\bigcup_{1\leq \alpha, \beta\leq q}  \left\{ J' \: \middle| \: J' = j\cup J(\alpha, \beta)\right\}
        \end{split}
    \end{equation*}
    which is precisely the index set in the first expression. Because $C_{\alpha, \beta}^j = -C_{\beta, \alpha}^j$ and $C_{\alpha, \alpha}^j =0$ by the antisymmetry of the Lie bracket, together with the following fact
    \begin{equation*}
        \eta_J = (-1)^{\alpha-1+\beta-2}\eta_{j_\alpha, j_\beta, j_1, \cdots, \hat{j_\alpha}, \cdots, \hat{j_\beta},\cdots, j_q} = (-1)^{\alpha + \beta -1} \eta_{\{j_\alpha, j_\beta\}\cup J(j_\alpha, j_\beta)},
    \end{equation*}
    and the claim is proven.
\end{proof}

Next we express the Hodge Laplacian $ \Delta := \delta \circ d + d\circ \delta$ as sum of four parts:
\begin{equation}\label{exp: hodge laplacian}
    \Delta =(d_\circ + d_{\wedge})(\delta_\circ + \delta_{\wedge})+ (\delta_\circ + \delta_{\wedge})(d_\circ + d_{\wedge}) = \Delta_{\circ} +\Delta_{\wedge} + \Delta_{\circ, \wedge} + \Delta_{\wedge. \circ}
\end{equation}
where:
\begin{align}
    \Delta_\circ           & = d_\circ \delta_\circ + \delta_\circ d_\circ \qquad
                           & \Delta_\wedge = d_\wedge \delta_\wedge + \delta_\wedge d_\wedge        \\
    \Delta_{\circ, \wedge} & = d_\circ \delta_\wedge + \delta_\wedge d_\circ \qquad
                           & \Delta_{\wedge, \circ} = d_\wedge \delta_\circ + \delta_\circ d_\wedge
\end{align}
For convenience, we choose $X_i$ to be an orthonormal basis of $ \mathfrak{g}$ with respect to the metric, and denote
\begin{equation*}
    \bog =  \sum_{i}X_i^2\in U( \mathfrak{g}_\complex ).
\end{equation*} 
We also note that $X_i$ acts via coadjoint representation on the basis of $\wedge^* \mathfrak{g}^*_\complex$ as
\begin{equation} \label{exp: abuse}
    \begin{split}
        \cadj(X_i)(\eta)_J 
        =  \sum_{1\leq u \leq q}\sum_{j\in I_{ \mathfrak{g} }}(-1)^{u-1} C^{j_u}_{j,i}\eta_{j\cup J(u)},
    \end{split}
\end{equation}
This identity comes from the expression \eqref{exp: const_struct2}. Define the \textbf{coadjoint* representation} as the dual representation of the coadjoint representation of $\mathfrak{g}$, i.e.: For all $X\in  \mathfrak{g}$:
\begin{equation}
    ((\cadj^*X)l, l')_{ \mathfrak{g}^*_\complex } = (l, (\cadj X)l')_{ \mathfrak{g}^*_\complex }
\end{equation}
This immediately extends to an action of $ \mathfrak{g}$ on $\wedge^* \mathfrak{g}^*$, which we also denote as $\cadj^*$. Note $\cadj^*$ and $-\ad$ are isomorphic representations by the following diagram:
    \begin{equation*}
        \begin{tikzcd}[row sep = large, column sep = large]
            \mathfrak{g}^*         & \mathfrak{g} \arrow[l, "{\flat, \cong}"swap ] \\
            \mathfrak{g}^* \arrow[r, "{\sharp, \cong}"swap] \arrow[u, "-\cadj_X^*"]  & \mathfrak{g} \arrow[u, "\ad_X"swap]
        \end{tikzcd}
    \end{equation*}
for all $X\in \mathfrak{g}$. Here $\flat$ and $\sharp$ are the musical isomorphisms.

Next we form a general expression for each component of the Hodge Laplacian. The following proposition holds for an arbitrary Lie algebra $ \mathfrak{g}$: 
\begin{proposition}\label{thm: rep_Laplacian}
    Given a hermitian $\mathfrak{g}$-module $(V, \tau)$ and form $C^q( \mathfrak{g}; V )$ as above. Then the Laplacian admits the following expressions:
    \begin{enumerate}[label = \textnormal{(\Alph*)}]
        \item $\Delta_\circ$ acts on $C^q( \mathfrak{g}; V )$ as
              \begin{equation}\label{exp: laplace_circ}
                (\Delta_\circ \eta)_J = -\sum_j \tau(X_j)^*\tau(X_j) \eta_J +\sum_{a}\tau(X_a)\sum_{1\leq u\leq q}(-1)^{u}C^a_{j, j_u} \eta_{\{j\}\cup J(u)};
              \end{equation}
        \item $\Delta_{\wedge}$ acts on $C^q( \mathfrak{g}; V )$ as
              \begin{equation}\label{exp: laplace_wedge}
                  \begin{split}
                    (\Delta_\wedge\eta)_J = \sum_{j} \sum_{\alpha< \beta}C^j_{\alpha, \beta}\Biggl(&\sum_{1\leq \gamma\leq q}(-1)^{\gamma}C^{j_\gamma}_{\alpha, \beta}\eta_{j\cup J(\gamma)}\\
                     + &\sum_{1\leq u < v\leq q}(-1)^{u+v} C^j_{j_u, j_v}\eta_{\{\alpha, \beta\}\cup J(u, v)}\Biggr);
                  \end{split}
              \end{equation}
        \item $\Delta_{\circ, \wedge}$ acts on $C^q( \mathfrak{g}; V )$ as $\sum_k\tau(X_k)^*\cadj(X_k)$;
        \item $\Delta_{\wedge, \circ}$ acts on $C^q( \mathfrak{g}; V )$ as $-\sum_k\tau(X_k)\cadj^*(X_k)$;
    \end{enumerate}
    In particular, if we take $V$ to be unitary, then the $0$-th Laplacian on $C^0( \mathfrak{g}; V )$ is seen to take the form:
    \begin{equation}
        \Delta_0 = -\tau(\bog).
    \end{equation}
\end{proposition}

\begin{proof}
    We prove the formula for $\Delta_{\circ, \wedge}$ as an example. The rest can be found in \cite[Proposition~1.5]{Thesis}.
    \begin{equation*}
        \begin{split}
            (\delta_\wedge d_\circ \eta)_J &=  \sum_{\alpha<\beta}\sum_{1\leq u \leq q-1} (-1)^{u-1}C^{j_u}_{\alpha, \beta}(d_\circ\eta)_{\{\alpha, \beta\}\cup J(u)} \\
            &= \sum_{\alpha<\beta}\sum_{1\leq u\leq q-1}(-1)^{u-1}C_{\alpha, \beta}^{j_u}\Bigg( \tau(X_\alpha)\eta_{\beta\cup J(u)} - \tau(X_\beta)\eta_{\alpha\cup J(u)} \\
            &\qquad\qquad  -\sum_{\substack{1\leq v\leq q \\ v\neq u}} (-1)^{v[u]}\tau(X_{j_v})\eta_{\{\alpha, \beta\}\cup J(u,v)}\Bigg).
        \end{split}
    \end{equation*}
    In the last step, the notation $v[u]$ is an abbreviation of notation, defined by
    \begin{equation}
        v[u] = \begin{cases}
            v   & \textnormal{ if }v < u \\
            v-1 & \textnormal{ if }u < v.
        \end{cases}
    \end{equation}
    On the other hand,
    \begin{align*}
        (d_\circ \delta_\wedge\eta)_J & =  \sum_{1\leq v\leq q+1}(-1)^{v-1} \tau(X_{j_v})(\delta_\wedge\eta)_{J(v)} \\
                                    & = \sum_{1\leq v\leq q+1}\sum_{\alpha<\beta} \sum_{\substack{1\leq u\leq q   \\ u\neq v}}(-1)^{v-1}\tau(X_{j_v})C^{j_u}_{\alpha, \beta}\cdot (-1)^{u[v]}\eta_{\{\alpha, \beta\}\cup J(u,v)}
    \end{align*}
    We see the last summand in $d_\circ\delta_\wedge\eta$ cancels with $\delta_\wedge d_\circ$ as
    \begin{equation*}
        (-1)^{v+u[v]} = -(-1)^{u+v[u]} 
    \end{equation*}
    for all $u\neq v$. Hence
    \begin{equation}
        \begin{split}
            (\Delta_{\circ, \wedge}\eta)_J &= \sum_{1\leq u \leq q}\sum_{\alpha< \beta}(-1)^{u-1}C^{j_u}_{\alpha, \beta}\left( \tau(X_\alpha)\eta_{\beta\cup J(u)}-\tau(X_\beta)\eta_{\alpha\cup J(u)} \right)\\
            &= \sum_{\gamma}(\tau(X_\gamma)\cadj(X_\gamma))(\eta_J),
        \end{split}
    \end{equation}
    where the last identity comes from our convention \eqref{exp: abuse}. Hence the third statement is proven. The rest is proved similarly. 
\end{proof}

\begin{remark} \label{rmk: square circ define as derivation}
    For later purposes, we note that the operator 
    \begin{equation*} \label{exp: square_circ}
        \begin{split}
            (\square_\circ \eta)_J := (\Delta_{\circ}\eta + \tau(\bog)\eta)_J = \sum_{a}\tau(X_a)\sum_{\substack{j \\ 1\leq u \leq q} }(-1)^{u}C^a_{j, j_u} \eta_{\{j\}\cup J(u)}
        \end{split}
    \end{equation*}
    acts on the $C^q( \mathfrak{g}; V )\cong  V \otimes \wedge^p  \mathfrak{g}^*_\complex $ as derivations on the exterior algebra part, in the following sense:
    \begin{equation}
        \square_\circ(\eta_J\omega^J) = \square_\circ (\eta_J \omega^{J_1})\wedge \omega^{J_2} + \omega^{J_1} \wedge \square_\circ (\eta_J\omega^{J_2})
    \end{equation}
    for any $\omega^J = \omega^{J_1}\wedge \omega^{J_2}$. This can be proved by comparing the expressions on both sides:
    \begin{equation*}
        \begin{split}
            &\square_\circ (\eta_J \omega^{J_1})\wedge \omega^{J_2} + \omega^{J_1} \wedge \square_\circ (\eta_J\omega^{J_2})\\
        =& \sum_{a} \tau(X_a)(\eta_{J})\biggl(  \sum_{\substack{j \\ 1\leq u\leq |J_1|}}(-1)^{u}  C^a_{j, j_u}\omega^{j}\wedge \omega^{J_1(u)}\wedge \omega^{J_2} \\
        &\ \ + \sum_{\substack{j \\ |J_1|+1\leq u \leq |J_2|}} (-1)^{u - |J_1|} C^{a}_{j, j_u}\omega^{J_1}\wedge \omega^j \wedge \omega^{J_2(u-|J_1|)}\biggr) \\
        =& \square_\circ(\eta_J\omega^J)
        \end{split}
    \end{equation*}
    where the last identity follows from $(-1)^{u - |J_1|}\omega^{J_1}\wedge\omega^j = (-1)^u \omega^j \wedge \omega^{J_1}$. Hence the claim is proven.
\end{remark}
If we assume further that $ \mathfrak{g}$ is real reductive, one can fix a non-degenerate $G$-invariant bilinear form $B(-,-)$ on $ \mathfrak{g}_\complex$ that restricts to the Killing form on $[ \mathfrak{g, g} ]$. If we fix a Cartan involution $\theta$ as in the previous section, we note $B(-,-)$ is negative definite on $ \mathfrak{k}$ and positive definite on $ \mathfrak{p}$. This gives us a positive definite $\Ad(K)$-invariant hermitian inner product, denoted as $B^\theta$, such that
\begin{equation}
    B^\theta(X, Y) := B(X, \theta \bar{Y})
\end{equation}
for all $X, Y\in  \mathfrak{g}_\complex$, where $\bar{\cdot}$ denotes the complex conjugation of $ \mathfrak{g}_\complex$ over $ \mathfrak{g}$.  In this case the previous $\bog$ admits a nicer expression:
\begin{equation}
    \bog  = \Omega_G - 2\Omega_K\in U( \mathfrak{g}_\complex ) 
\end{equation} 
Recall $\Omega_G$ is the Casimir element in the center $Z( \mathfrak{g}_\complex )$ of $U( \mathfrak{g}_\complex )$ induced by the $G$-bi-invariant bilinear form $B$:
\begin{equation*}
    \Omega_G = \sum_{i}X_i X^i
\end{equation*}
for a basis $\{X_i\}$ of $ \mathfrak{g}$ with the dual basis $\{X^i\}$ with respect to $B$. Here $\Omega_K$ is defined in a similar way by replacing the basis of $ \mathfrak{g}$ by that of $ \mathfrak{k}$. If we take $\{X_i\}$ to be an orthonormal basis of $ \mathfrak{g}$ with respect to $B^\theta$, then $\bog = \sum_{i}X_i^2$ is just the sum of squares.

Fix an orthonormal basis of $ \mathfrak{g}$ 
\begin{equation*}
    \{ X_1, \cdots, X_m, Y_1, \cdots, Y_{n-m}\}
\end{equation*} 
with respect to the Killing form $B^\theta$, such that $\{X_i\}$ is an the orthonormal basis of $ \mathfrak{k}$ and $\{Y_\alpha\}$ that of $\mathfrak{p}$. We denote the index sets of $ \mathfrak{k}$ and $ \mathfrak{p}$ as $I_{\mathfrak{k}}$ and $I_ \mathfrak{p}$ respectively, and again $I_{\mathfrak{g}} = I_{ \mathfrak{k}} \sqcup I_{\mathfrak{p}}$ the index set of $ \mathfrak{g}$. Before the proof of the corollary, we remark that the coadjoint and coadjoint* representations satisfy the following relations when restricted to $ \mathfrak{k}$ and $ \mathfrak{p}$ respectively:
\begin{equation} \label{exp: cadj vs adj}
    \cadj|_{ \mathfrak{k} } = \cadj^*|_{ \mathfrak{k} } \qquad \cadj|_{ \mathfrak{p} } = -\cadj^*|_{ \mathfrak{p} }.
\end{equation}
\begin{corollary}[\textbf{Generalized Kuga's Lemma}]\label{thm: generalized_Kuga}
    Assume $\mathfrak{g}$ is reductive with a Cartan decomposition $ \mathfrak{g} =  \mathfrak{k} \oplus  \mathfrak{p}$ as above. Let $V$ be an hermitian $\mathfrak{g}$-module. Define a bigrading on $C^k( \mathfrak{g}; V ) = \sum_{p+q=k}C^{p,q}( \mathfrak{g}; V )$ with
    \begin{equation*}
        C^{p,q}( \mathfrak{g}; V ) := \Hom_\complex(\wedge^p \mathfrak{k}\otimes \wedge^q{\mathfrak{p}},V) \cong V\otimes \wedge^p  \mathfrak{k}^* \otimes \wedge^q  \mathfrak{p}^*.
    \end{equation*}
    Then $\Delta_\circ$ and $\Delta_\wedge$ take the following simplified form:
    \begin{enumerate}[label = \textnormal{(\Roman*)}]
        \item $\Delta_\circ = \tau(\bog) + \square_\circ$ acts on $C^{p, 0}( \mathfrak{g}; V )$ and $C^{0, q}( \mathfrak{g}; V )$ by
              \begin{equation}
                      \begin{split}
                        \square_\circ \Big|_{C^{p,0}( \mathfrak{g}; V )}&\cong +\sum_{j\in I_ \mathfrak{g} } \tau(X_j)\cadj^*(X_j) \\
                        \square_\circ\Big|_{C^{0, q}( \mathfrak{g}; V )}&\cong  - \sum_{j\in I_ \mathfrak{g} }\tau(X_j)\cadj^*(X_j)
                      \end{split}
              \end{equation}
             This extends to an action on $C^{p,q}( \mathfrak{g}; V )$ by derivation, as a consequence of \Cref{rmk: square circ define as derivation}.
        \item $\Delta_\wedge$ acts on $C^n( \mathfrak{g}; V )$ by $\sum_{j} \cadj^*(X_j)\cadj(X_j)$.
    \end{enumerate}
    In particular, if we abbreviate the action by:
    \begin{equation}
        \begin{split}
            A_ \mathfrak{g}  &= -\tau(\bog):C^{p,q}( \mathfrak{g}; V ) \longrightarrow C^{p,q}( \mathfrak{g}; V)\\
            B_ \mathfrak{k}  &= \sum_{k \in I_ \mathfrak{k} }\tau(X_k)\cadj^*(X_k): C^{p,q}( \mathfrak{g}; V ) \longrightarrow C^{p, q}( \mathfrak{g}; V )\\
            C_ \mathfrak{p} & =  \sum_{\alpha\in I_ \mathfrak{p} }\tau(Y_\alpha)\cadj^*(Y_\alpha): C^{p,q}( \mathfrak{g}; V )\longrightarrow C^{p-1, q+1}( \mathfrak{g}; V )\oplus C^{p+1, q-1} ( \mathfrak{g}; V )\\
            D_{ \mathfrak{g} }& = \frac{1}{2}\cadj^*(\Omega_G):C^{p,q}( \mathfrak{g}; V ) \longrightarrow C^{p,q}( \mathfrak{g}; V)
        \end{split}
    \end{equation}
    then each component of $\Delta_1$ acts on $1$-forms $C^1( \mathfrak{g}; V ) \cong (V\otimes  \mathfrak{k}^* )\oplus (V\otimes  \mathfrak{p}^*) $ via the following operator-valued matrices:
    \begin{equation}
        \begin{split}
            \Delta_\circ|_{C^1( \mathfrak{g}; V  )} = \begin{pmatrix}
                A_ \mathfrak{g} +B_ \mathfrak{k}  &-C_ \mathfrak{p}\\
                C_{ \mathfrak{p} }  & A_ \mathfrak{g} - B_ \mathfrak{k} 
            \end{pmatrix} &\qquad \Delta_\wedge|_{C^1( \mathfrak{g}; V  )} = \begin{pmatrix}
                D_{ \mathfrak{g} }  &                   \\
                                  & D_{ \mathfrak{g} }
            \end{pmatrix} \\
            \Delta_{\circ, \wedge}|_{C^1( \mathfrak{g}; V  )} = \begin{pmatrix}
                -B_ \mathfrak{k}  & C_ \mathfrak{p} \\
                C_ \mathfrak{p}  & -B_ \mathfrak{k}
            \end{pmatrix} &\qquad \Delta_{\wedge, \circ}|_{C^1( \mathfrak{g}; V  )} = \begin{pmatrix}
                -B_ \mathfrak{k}  & -C_ \mathfrak{p} \\
                -C_ \mathfrak{p}  & -B_ \mathfrak{k}
            \end{pmatrix}
        \end{split}
    \end{equation}
    Summing all terms up, the $1$-th Laplacian on $C^1( \mathfrak{g}; V )$ is seen to take the form:
    \begin{equation}
       \Delta_1 = \begin{pmatrix}
            A_ \mathfrak{g}  -B_ \mathfrak{k} +D_{ \mathfrak{g} } & -C_ \mathfrak{p}                   \\
            C_ \mathfrak{p}                   & A_ \mathfrak{g}  -3B_ \mathfrak{k} +D_{ \mathfrak{g} }
        \end{pmatrix}
        \label{exp: 1_laplacian block}
    \end{equation}
\end{corollary}

\begin{proof}
    The details of the proof can be found in \cite[Corollary~1.7]{Thesis}. The main ingredient is to refine the results from \Cref{thm: rep_Laplacian} and by exploiting the extra symmetries in the structural constants: Given $X_i, X_j, X_k \in  \mathfrak{k}$ and $Y_\alpha, Y_\beta\in  \mathfrak{p} $ some basis vectors, one readily verify the following identity:
    \begin{equation*}
        C^{\beta}_{\alpha, i} = C^{\alpha}_{i, \beta}= -C_{\beta, \alpha}^i \qquad C^k_{i, j} = C^i_{j, k}
    \end{equation*}
    whereas all structural constants $C^{i}_{\alpha, j}$ vanish by orthogonality. 
    
    Assuming $J\subseteq I_\mathfrak{k} $, the coadjoint action for fixed $a$ can be rewritten as:
    \begin{align*}
        (\square_\circ \eta)_J &= \sum_{a\in I_ \mathfrak{g} }\tau(X_a)\sum_{\substack{j \\ 1\leq u \leq q} } (-1)^{u-1}C^a_{j, j_u} \eta_{j\cup J(u)} \\
                &= -\sum_{a\in I_ \mathfrak{k} }\biggl(\tau(X_a)(-1)^u\sum_{\substack{1\leq u \leq q \\j\in I_\mathfrak{k}} }C^a_{j, j_u}\eta_{j\cup J(u)}\biggr)\\
                 &\qquad \qquad  - \sum_{\alpha\in I_ \mathfrak{p}}\biggl(\tau(X_\alpha)(-1)^u\sum_{\substack{1\leq u \leq q \\\beta\in I_\mathfrak{p}} }C^\alpha_{\beta, j_u}\eta_{j\cup J(u)}\biggr)\\
                &= \sum_{a\in I_ \mathfrak{k} }\biggl(\tau(X_a)(-1)^u\sum_{\substack{1\leq u \leq q \\j\in I_\mathfrak{k}} }C^j_{a, j_u}\eta_{j\cup J(u)}\biggr)\\
                &\qquad \qquad + \sum_{ \alpha\in I_ \mathfrak{p}}\biggl(\tau(X_\alpha)(-1)^u\sum_{\substack{1\leq u \leq q \\ \beta\in I_\mathfrak{p}} }C^\beta_{\alpha, j_u}\eta_{j\cup J(u)}\biggr)\\
                &= \sum_{a\in I_ \mathfrak{g} }\tau(X_a)\cadj(X_a)(\eta_{J}),
    \end{align*}
   where in the last identity we used $\cadj X_\alpha (\omega^\beta) = \sum_{\gamma} C_{\gamma, \alpha}^\beta \omega^\gamma$, as well as our convention \eqref{exp: abuse}, and then by exploiting the identities of the structural constants above. The rest of the identities are derived in a similar way.
\end{proof}

\section{Elliptic operators on Lie groups} \label{section: elliptic operators}

In this section we review and prove some classical results about kernel estimates of the heat semigroup. The strategy is to first develop a Gaussian estimate for the second order differential operator $\bog$. 

We first derive a Gaussian upper bound for strongly elliptic operators with constant coefficients. Given a continuous representation $(\pi, H_\pi)$ of $G$, i.e, a homomorphism $\pi$ from $G$ to the group of linear transformation on $H_\pi$ such that the map
\begin{equation*}
    G\times H_\pi \to H_\pi \qquad (g, v)\mapsto \pi(g)v
\end{equation*}
is continuous. We can choose a basis $\{X_i\}$ of $\mathfrak{g}$ and define \emph{differential operator with constant coefficients} as a polynomial in $D_i = \pi(X_i)$, where:
\begin{equation} \label{exp: differential operator}
    D = \sum_{I} c_ID^I
\end{equation} 
with $D^I = D_{i_1}D_{i_2}\dots D_{i_n}$ for $I=(i_1, \cdots, i_n)$, and $c_I$ are complex coefficients. We say the differential operator is \emph{strongly elliptic} if its associated principal symbol satisfies the following inequality
\begin{equation*}
    \operatorname{Re}\left( (-1)^{\frac{m}{2}}\sum_{|I| = m}c_I\xi^I \right)>0
\end{equation*}
for all $\xi\in \real^d \backslash \{0\}$, where $m$ is the order of the differential operator. Furthermore, we denote the $C^n$-vectors in the representation $H_\pi$ as
\begin{equation}
    C^n(H_\pi) = \bigcap_{1\leq i_1, \dots, i_n\leq d} \operatorname{Dom}(\pi(X_{i_1}\dots X_{i_n})),
\end{equation}
which defines a subspace of the smooth vectors of $H_\pi$ independent of the choice of basis \cite[Page~8]{RobinsonBook1991}. The $C^n$-norm of $v\in H_\pi$ is defined as
\begin{equation}  \label{exp: Ck-norm}
    \left\lVert v \right\rVert _{C^n(H_\pi)} =\sup_{|I|\leq n} \left\lVert \pi(X^I) v \right\rVert _{H_\pi}
\end{equation}
and the analytic vectors of $H_\pi$ are defined to be those $v\in H_\pi$ such that
\begin{equation*}
    \sum_{n\geq 1}\frac{t^n}{n!}\left\lVert v \right\rVert _{C^n(H_\pi)} < \infty
\end{equation*}
for some $t>0$. It is a classical result proved independently by Langlands \cite{langlandsthesis} and Nelson \cite{Nelson:1959a} that the analytic vectors are dense in the continuous representation $H_\pi$ of $G$. We collect some results of Langlands' thesis which we will use in establishing the Gaussian estimate. The details can be found in \cite{RobinsonBook1991}.

Let $(\pi, H_\pi)$ be weakly* or strongly continuous representation, and let $D$ be a strongly elliptic operator associated with $\pi$. Then $D$ is closable and its closure $\bar{D}$ generates a holomorphic semigroup, denoted as $e^{-t\bar{D}}$ on respective representations satisfying the property:
\begin{equation*}
    \bar{D}v = \lim_{t\to 0} \frac{v - e^{-t\bar{D}}v}{t}
\end{equation*}
for all $v\in \mathrm{Dom}(D)$. Moreover, the semigroup satisfies the following the property \cite[Theorem~1.5.1]{RobinsonBook1991}:
\begin{enumerate}[label = \Roman*.]
    \item $e^{-t\pi(\bar{D})}H_\pi\subseteq H_\pi^\infty$ maps the space into its smooth vectors $H_\pi^\infty$ for all $t>0$;
    \item The map $z\mapsto e^{-z\pi(\bar{D})}$ defines a holomorphic family of bounded operators on $H_\pi$ in the sector $\{z\in \complex: |\arg z|<C\}$ for some $C\in (0, \pi/2]$ depending on the elliptic operator. Moreover, $e^{-t\pi(\bar{D})}H_\pi\subseteq \operatorname{Dom}(\bar{D})$ for $t>0$ and 
    \begin{equation}\label{exp: coarse holomorphy bound}
        \left\lVert \pi(\bar{D}) e^{-t\pi(\bar{D})} v \right\rVert _{H_\pi} \leq M t^{-1} \left\lVert v \right\rVert _{H_\pi}
    \end{equation}
    for some $M\geq 1$.
\end{enumerate}
Moreover, there exists $k, l>0$ such that we have the following small time estimate \cite[Theorem~II.2.2]{RobinsonBook1991}:
\begin{equation}
    \label{exp: robinsonII.2.1}
    \left\lVert e^{-t\pi(\bar{D})}v \right\rVert _{C^n(H_\pi)} \leq kl^n n! t^{-\frac{n}{m}} \left\lVert v \right\rVert _{H_\pi} 
\end{equation}
for all $v\in H_\pi$ and $t\in (0, 1]$. On the other hand, using the fact that $e^{-t\bar{D}}$ is a holomorphic family of bounded operators for $t>0$, and general semigroup theory \cite[p.96]{RobinsonBook1991}, we choose the lower bound of the operator norm
\begin{equation*}
    \omega = \inf_{t>0} \frac{\log \| e^{-t\bar{D}} \| }{t} > -\infty.
\end{equation*} 
Here $\left\lVert \cdot \right\rVert $ denotes the operator norm on $H_\pi$. Consequently there is a $M\geq 1$, such that
\begin{equation}
    \label{exp: continuity bound}
    \left\lVert e^{-t\bar{D}}\phi \right\rVert _{H_\pi}\leq M e^{\omega t}\left\lVert \phi \right\rVert _{H_\pi}
\end{equation}
for all $t>0$. We call this the continuity bound.

Before moving on to establish the Gaussian estimate, we introduce the last bound which is a Sobolev inequality in the Lie group setting. If we choose the representations $(\pi, H_\pi) = (L, L^p(G))$ with $L$ the left translation of $G$ on $L^p(G)$, then we can define the Sobolev spaces 
\begin{equation*}
    W^{k, p}(G):= C^k(L^p(G)) 
\end{equation*}
as the $C^k$-vectors of $L^p(G)$ with $C^k$-norm in $L^p(G)$ defined as in \eqref{exp: Ck-norm}, and correspondingly the $C^k(G)$ as the $C^k$-vectors of $C_0(G)$, the continuous functions on $G$ that vanish at infinity with:
\begin{equation*}
    C^k(G) := C^k(C_0(G)) \quad \left\lVert \varphi \right\rVert _{C^m} = \sup_{|I|\leq m}\left\lVert L(X^I)\varphi \right\rVert _{L^\infty(G)}.
\end{equation*}
Let $U$ be a bounded open subset of the identity component of $G$. Denote $d:=\dim G$, and the Sobolev inequality is proved essentially the same way as the real case \cite[Lemma~III.2.3]{RobinsonBook1991}:
\begin{equation} \label{exp: sobolev embedding}
    \left\lVert \varphi \right\rVert _{C^m(U)} \leq C\left\lVert \varphi \right\rVert _{W^{k, p}(U)} 
\end{equation}
for all $0\leq m < k - \frac{d}{p}$ and for all $\varphi \in W^{k, p}(U)$, with $C$ being the implied constant.  

Next we consider the heat kernel associated with $D$. Fix a left $G$-invariant metric on $G$. This defines a distance function $|g|:=d(g, e_G)$ on $G$, as well as a left Haar measure on $G$. Then there is integrable function $k_t: \real_+\times G \to \complex$ such that for every strongly continuous and weakly* continuous representation $(H_\pi, \pi)$ of $G$, the corresponding heat semigroup $e^{-t\pi(\bar{D})}$ defined above satisfies
\begin{equation}
    e^{-t\pi(\bar{D})} = \int_G\pi(g)k_t(g)\mass{g}.
\end{equation}
with the kernel $(t,g)\mapsto k_t(g)$ defines an analytic function on $\real_+\times G$ \cite[Theorem~III.2.1]{RobinsonBook1991}.

For $1\leq p\leq \infty$, we define the weighted $L^p$-space $L^p_{\rho}$ for any $\rho \geq 0$ as
\begin{equation*}
    L^p_\rho := L^p(G; e^{\rho|g|}\mass{g}) \qquad \left\lVert \varphi \right\rVert _{L^p_\rho} := \left( \int \Big|\varphi(g) e^{\rho|g|} \Big|^p\mass{g} \right)^{1/p}
\end{equation*} 
and let $L^\rho$ be the left regular representation of $G$ on it. All weighted $L^p$-spaces are Banach algebras under convolution \cite[Section~I]{KUZNETSOVA2012}. Similarly, define the weighted $L^\infty$-space $L^\infty_\rho(G)$ with weighted $L^\infty$-norm as
\begin{equation}
    \left\lVert \varphi \right\rVert _{L^\infty_\rho} := \operatorname*{ess.sup}_{g\in G}e^{\rho|g|}|\varphi(g)|.
\end{equation}
We now derive a pointwise Gaussian bound from the inequalities we listed above. The representation in question is $(L^\rho, L^p_\rho(G))$ for suitably chosen $\rho$. The following lemma is stated in \cite[Corollary~III.2.5]{RobinsonBook1991}. We include the proof to complete the discussion.
\begin{lemma} \label{lemma: RobinsonIII.2.5}
    Let $G$ be a Lie group, and $X_1, \dots, X_d$ a basis of the Lie algebra $ \mathfrak{g}$. Given a $G$-invariant strongly elliptic operator $D$ of order $m$ on $G$ expressed in $X_i$, and form the kernel $k_t$ for $t> 0$ as above. Then for each $\rho\geq 0$ there exist $a, b, c>0$ and $\omega\geq 0$ such that we have the following pointwise Gaussian estimate:
    \begin{equation}
        \label{exp: RobinsonIII.2.5}
        \left\lvert L(X^\alpha)\partial_t^\ell k_t(g) \right\rvert \leq ab^{|\alpha|}c^\ell |\alpha|!\ell! (1+t^{-(\ell+\frac{|\alpha|+d+1}{m})})e^{\omega t}e^{-\rho|g|}
    \end{equation}
    for all $g>G$ and $t>0$. Here again $L(X^\alpha)$ denotes the left regular representation of $G$ and  $\partial_t$ denotes the differential in the time direction.
\end{lemma}

\begin{proof}
    For simplicity we prove the statement only for unimodular $G$ and comment on the other case in the end of the proof. We consider the representation to be the left-regular translation $L^\rho$ of $G$ on the weighted $L^1$-space $L_\rho^1(G)$.
    
    Let $e^{-tL^\rho(\bar{D})}$ be the corresponding heat semigroup and:
    \begin{equation}
        L^\rho(X^\alpha) D^\ell e^{-tL^\rho(\bar{D})} \varphi(e_G) = \int_G \left( L(X^\alpha)(-\partial_t)^\ell k_t \right)(g)e^{\rho|g|}\cdot(\varphi(g^{-1})e^{-\rho|g|})\mass{g}.
    \end{equation}
    Here we use the fact that $Dk_t +\partial_tk_t = 0$, and we can interchange the order of the derivative and the integral via the following equality
    \begin{equation*}
        D(k_t * \varphi) = (Dk_t)*\varphi,
    \end{equation*}
    again using Langlands' result that $k_t\in C^\infty(G)\cap L^1(G)$. Next by the duality between $L^1$ and $L^\infty$, we rewrite the integral as
    \begin{equation*}
        \sup_{g\in G} |(\partial_g^\alpha\partial_t^\ell k_t)(g)|e^{\rho|g|} = \sup_{\left\lVert \varphi \right\rVert _ {L^1_\rho}\leq 1 }\left\{ |L^\rho(X^\alpha) D^\ell e^{-tL^\rho(\bar{D})} \varphi(e_G)| \right\}.
    \end{equation*}
    Now consider the $C^k$-vectors of the weighted $L^1$-space
    \begin{equation*}
        W^{k, 1}_\rho(G):= C^k(L^1_\rho(G)),
    \end{equation*}
    with weighted $C^k$-norm $\left\lVert \cdot \right\rVert _{C^k_\rho}$ defined as in \eqref{exp: Ck-norm} by replacing $H_\pi$ by $L^1_\rho(G)$. Consider their restriction to a bounded open neighborhood $U$ of $e_G$, where the $L^1$-norm and the weighted $L^1$-norm are equivalent. Applying now the Sobolev embedding \eqref{exp: sobolev embedding} by taking $m, k, p = 0, d+1, 1$ we get a constant $c_{\rho, U}>0$ such that:
    \begin{equation*}
        \begin{split}
            |X^\alpha D^\ell e^{-tL^\rho(\bar{D})} \varphi(e_G)| &\leq \left\lVert X^\alpha D^\ell e^{-tL^\rho(\bar{D})\varphi(e_G)} \right\rVert _{C^0(U)}\\
             &\leq c_{\rho, U} \left\lVert X^\alpha D^\ell e^{-tL^\rho(\bar{D})} \varphi \right\rVert _{W^{d+1, 1}_\rho(U)}\\
            &\leq c_{\rho,U}\left\lVert D^\ell e^{-tL^\rho(\bar{D})} \varphi \right\rVert _ {W^{|\alpha|+d+1, 1}_\rho(U)}.
        \end{split}
    \end{equation*}
    Hence $ |\partial_g^\alpha\partial_t^\ell k_t(g)|$ remains bounded for all $g\in G$ and $t>0$ by the operator norm 
    \begin{equation*}
        |\partial_g^\alpha\partial_t^\ell k_t(g)| \leq c_{\rho, U}e^{-\rho|g|}\left\lVert D^\ell e^{-tL^\rho(\bar{D})} \right\rVert _{L^1_\rho\to W^{|\alpha|+d+1, 1}_\rho(U)}.
    \end{equation*}
    with $\left\lVert \cdot \right\rVert _{X\to Y}$ denotes the operator norm between Banach spaces $X$ and $Y$. Again from the fact that $\phi = e^{-tL^\rho(\bar{D})}\varphi$ solves the equation $\partial_t\phi +\bar{D}\phi=0$, we obtain the following identity:
    \begin{equation*}
        D^\ell e^{-tL^\rho(\bar{D})} = e^{-t_1L^\rho(\bar{D})}(D e^{-\frac{t_2}{\ell} L^\rho(\bar{D})})^\ell e^{-t_3L^\rho(\bar{D})}
    \end{equation*}
    for every triplet $t_1, t_2, t_3>0$ with $t_1+t_2+t_3 = t$. Rewrite $\beta :=|\alpha|+d+1$, we therefore want to estimate the following:
    \begin{equation*}
        \begin{split}
            &e^{\rho|g|}\left\lvert (\partial_g^\alpha\partial_t^\ell k_t(g)) \right\rvert \\
            \leq & c_{\rho, U}\left\lVert e^{-t_1L^\rho(\bar{D})} \right\rVert _{L^1_\rho\to W^{\beta, 1
            }} \left( \left\lVert D e^{-\frac{t_2}{\ell} L^\rho(\bar{D})} \right\rVert _{L^1_\rho\to L^1_\rho} \right)^\ell \left\lVert e^{-t_3 L^\rho(\bar{D})} \right\rVert _{L^1_\rho\to L^1_\rho}.
        \end{split}
    \end{equation*}
    The norms in three time intervals are bounded via different methods:
    \begin{enumerate}[label = (\Roman*)]
        \item For $t_1\in (0, 1]$, we use the small time estimate \eqref{exp: robinsonII.2.1}. Applying to $L^1_\rho$-norm to estimate the $W^{\beta,1}$-norm:  
        \begin{equation}
            \left\lVert e^{-t_1L^\rho(\bar{D})} \right\rVert _{L^1_\rho \to W^{\beta, 1
            }}\leq c'c^\beta \beta!t^{-\frac{\beta}{m}}
        \end{equation} 
        for some $c', c>0$. This contributes the factor $a'b^{|\alpha|} |\alpha|!t^{-\beta/m}$ to the final expression;
        \item For $t_2\in (0, \ell]$, recall that $e^{-tL^\rho(\bar{D})}$ defines a holomorphic family of bounded operator for $t$ in the real line, and $e^{-tL^\rho(\bar{D})}L^1_\rho \subseteq \operatorname{Dom}(\bar{D})$. Hence we have the following bound:
            \begin{equation}
                  \left\lVert De^{-\frac{t_2}{\ell}L^\rho(\bar{D})} \right\rVert _{L^1_\rho\to L^1_\rho} \leq c_2{\frac{\ell}{t_2}}
            \end{equation}
            for some $c_2 >0$, by changing the variable $t_2$ to $t_2/\ell$ and recalling the bound \eqref{exp: coarse holomorphy bound}. Hence the second factor contributes a term $\ell^\ell c_2^\ell t_2^{-\ell}\sim \ell! c_2^\ell t_2^{-\ell}$ to the final expression by Stirling's approximation;
        \item Lastly the continuity bound \eqref{exp: continuity bound} gives 
        \begin{equation*}
            \left\lVert e^{-t_3L^\rho(\bar{D})} \varphi\right\rVert _{L^1_\rho} \leq C''e^{\omega t_3}\left\lVert \varphi \right\rVert _{L^1_\rho},
        \end{equation*}
        and consequently $\left\lVert e^{-t_3L^\rho(\bar{D})} \right\rVert _{L^1_\rho\to L^1_\rho} \leq C''e^{\omega t_3}$ for some $C''>0$. Hence the third factor contributes a term $e^{\omega t_3}$ to the final expression.
    \end{enumerate}
    Summing up all the contributions we see the upper bound of the derivative is indeed as claimed.

    We end the proof with a comment on generalizing the proof to non-unimodular cases. All the key ingredients are the same, but one needs to take into account the difference between the right and the left-invariant Haar measure, and also to measure the growth of the modular function $\Delta(g)$ resulting from changing one measure to the other. The modular function is an analytic homomorphism and hence we can bound it pointwise by $e^{\omega'|g|}$ for some $\omega'>0$ and $C\geq 1$. Replacing $\rho$ in the argument by $\rho+\omega'$ one then gets the correct estimate.
\end{proof}

\begin{remark}
    The Gaussian bound we obtained in \Cref{lemma: RobinsonIII.2.5} is in fact crude and can be refined via techniques such as logarithmic Sobolev inequalities or Nash inequalities (see \cite{RobinsonBook1991} for an extensive survey). But for our purpose, the current estimate is sufficient.
\end{remark}

\section{Proof of the Schwartz estimate} \label{section: proof of the Schwartz estimate}

In this section we prove \Cref{thm: Schwartz kernel}, the major result of this paper. First we recall the theory of bounded perturbation of heat semigroups. The following theorem is from \cite[Section~13.4]{HIllePhillips74}. 

\begin{theorem} \label{thm: bounded perturbation}
    Let $X$ be a Banach space. Let $e^{-tA}$ a continuous semigroup of bounded operators on $X$ with $A$ be the infinitesimal generator of a semigroup, and $B$ a linear operator such that $ \operatorname{Dom}(B) \supseteq \operatorname{Dom}(A)$, and $Be^{-tA}$ defines a $t$-measurable family of bounded operators on $X$. Moreover, we assume that both $e^{-tA}$ and $Be^{-tA}$ satisfy the following estimates:
    \begin{itemize}
        \item For some $M, \omega >0$, $\left\lVert e^{-tA} \right\rVert \leq  Me^{t\omega}$ for all $t>0$.
        \item For some $\alpha <1$ and $c>0$, $\left\lVert Be^{-tA} \right\rVert \leq ct^{-\alpha}$ for all $t\in (0, 1]$.
    \end{itemize}
    Here $\left\lVert \cdot \right\rVert$ denotes the operator norm on $X$. Then $A+B$ also generates a strongly continuous semigroup $e^{-t(A+B)}$ which admits the \emph{Dyson-Phillips expansion}, which is an absolutely convergent series for all $t>0$: 
    \begin{equation} \label{exp: Dyson-Phillips expansion}
        e^{-t(A+B)}f = \sum_{k=0}^\infty B\operatorname{Per}^k(e^{-tA}f)
    \end{equation}
    for all $f\in X$. Here $\operatorname{Per}^i(u(t))$ is defined recursively as:
    \begin{equation}
        \operatorname{Per}^0(u(t)) = u \quad \operatorname{Per}^k(u(t)) = \int_{0}^t e^{-(t-s)A}B \operatorname{Per}^{k-1}(u(s))ds.
    \end{equation}
\end{theorem}
The proof is modelled after a more general argument from \cite[Chapter~13]{HIllePhillips74}. We include a proof here, as the main ingredients of the proof will be used again in proving \Cref{prop: perturbed kernel estimates}. We begin with proving that the Dyson-Phillips expansion defines an absolutely convergent series, hence $\{e^{-t(A+B)}: t>0\}$ indeed forms a strongly continuous semigroup. This can be done by majorizing the series by the following two functions
\begin{equation}
    \phi(t) := \left\lVert e^{-tA} \right\rVert, \qquad \psi(t) := \left\lVert Be^{-tA} \right\rVert.
\end{equation} 
Both functions are non-negative and measurable functions in $t$, and they satisfy two conditions:
\begin{enumerate}[label = Property \arabic*:]
    \item $\int_0^1 \phi(t)+\psi(t)\mass{t} <\infty$;
    \item By the semigroup property of $e^{-tA}$, $\psi(t)$ satisfies the inequality: For all $t, s>0$,
    \begin{equation}
        \psi(t+s) \leq \psi(t) \phi(s).
    \end{equation}
\end{enumerate}
These two properties alone give some quantitative bounds on $\phi$ and $\psi$:
\begin{enumerate}[wide = 0pt, label = (\Roman*)]
    \item The finite sub-multiplicative function $\phi(t)$ is bounded on each interval of the form $(\epsilon, 1/\epsilon)$ by \cite[Theorem~7.4.1]{HIllePhillips74}. Together with Property 2, we see $\psi(t)$ is also bounded there;
    \item $\lim_{t\to \infty}t^{-1}\log \phi(t)$ exists, by the fact that $\phi$ is sub-multiplicative \cite[Theorem~7.6.1]{HIllePhillips74}. Denote this limit by $\omega_0$. This also gives an upper bound on $\psi$ by Property 2:
          \begin{equation} \label{exp: Hille13.4.4}
              \limsup_{t\to \infty}t^{-1}\log \psi(t) \leq \omega_0
          \end{equation}
    \item For any $\omega>\omega_0$ we have $\int_0^\infty e^{-t\omega}(\phi(t)+\psi(t))\mass{t}=:M_\omega<\infty$. This is immediate from the first property and \eqref{exp: Hille13.4.4}. Using the property 2 with the growth estimate of $\phi(t)$,
          \begin{equation*}
              0\leq 4e^{-t\omega} \psi(t) \leq (e^{-\omega(t-t_1)}\phi(t-t_1) + e^{-t_1\omega}\psi(t_1))^2
          \end{equation*}
          by the sub-additivity inequality. Hence
          \begin{align*}
              t(e^{-t\omega}\psi(t))^{1/2}
               & = 2\int_0^{t/2} (e^{-t\omega}\psi(t))^{1/2}\mass{t_1}                                               \\
               & \leq \int_0^{t/2} e^{-\omega(t-t_1)}\psi(t-t_1) + e^{-t_1\omega}\phi(t_1)\mass{t_1} \leq M_\omega,
          \end{align*}
          and we obtain the upper bound
          \begin{equation} \label{exp: Hille13.4.5}
              \psi(t) \leq e^{t\omega}t^{-2}M_\omega^2
          \end{equation}
          for all $t>0$.
\end{enumerate}
The goal is to bound $\operatorname{Per}^k(e^{-tA}f)$ in the Dyson-Phillips series by the convolution product $\phi * \psi^{*k}$, where $\psi^{*k}$ denotes the $k$-th convolution product of $\psi$. The series of convolution products can be estimated as follows \cite[Lemma~13.4.3]{HIllePhillips74}:
\begin{lemma} \label{lemma: Hille13.4.3}
    Suppose $\psi_0$ and $\psi_1$ are two nonnegative measurable functions satisfying property 1 and 2 as stated above. Then the series $\theta(t):= \sum_{k=0}^\infty(\psi_0 * \psi_1^{*k})(t)$ converges uniformly with respect to $t$ in the interval of $(\varepsilon, 1/\varepsilon)$ for $\varepsilon\in (0,1)$. Moreover, if $\omega>\omega_0>0$ is such that $\int_{0}^{\infty}e^{-t\omega}\psi_1(t) \mass{t}\leq 1$, then $\int_0^\infty e^{-t\omega}\theta(t)\mass{t}< \infty$.
\end{lemma}

\begin{proof}
    By the first quantitative bound above we see that both $\phi(t)$ and $\psi(t)$ are bounded on each interval of the form $(\epsilon, 1/\epsilon)$. Choose $\omega_1 >\omega_0$ so that:
    \begin{equation*}
        \int_0^\infty e^{-\omega_1\xi}(\phi(t)+\psi_0(t))\mass{t} \leq 1; \quad \int_0^\infty e^{-\omega_1\xi}(\phi(t)+\psi_1(t))\mass{t} \leq \frac{1}{16}.
    \end{equation*}
    Now from \eqref{exp: Hille13.4.5} we get $\psi_0(t)\leq t^{-2}e^{t\omega_1}$ and $\psi_1(t)\leq \frac{1}{16}t^{-2}e^{t\omega_1}$. By induction we will establish:
    \begin{equation} \label{exp: Hille13.4.7}
        \psi_0*\psi_1^{*n}(t)\leq 2^{-n}t^{-2}e^{t\omega_1}
    \end{equation}
    It suffices to estimate the constant for the induction step using the quantitative bounds we established before the lemma:
    \begin{align*}
             & (\psi_0*\psi_1^{*n})(t)
        \\
        =    & e^{t\omega_1}\left( \int_0^{t/2} + \int_{t/2}^t \left( e^{-(t-s)\omega_1}\psi_0*\psi_1^{*(n-1)}(t-s) \right) \cdot \left( e^{-s\omega_1}\psi_1(s) \right)\mass{s} \right)               \\
        \leq & e^{t\omega_1} \left( 2^{-(n-1)}2^2  t^{-2}\int^{t/2}_0 e^{-s\omega_1}\psi_1(s)\mass{s} + \frac{2^2}{16} t^{-2} \int_0^{t/2}e^{-s\omega_1}(\psi_0*\psi_1^{*(n-1)})(s)\mass{s} \right) \\
        \leq & 2^{-n}t^{-k}e^{-t\omega_1},
    \end{align*}
    with the two inequalities coming from the induction hypothesis. Having established \eqref{exp: Hille13.4.7} we see $\theta_t$ is an absolutely convergent series, as it is majorized by the uniformly convergent series $\sum_n 2^{-n}t^{-2}e^{t\omega_1}$. This proves the first claim. For the second claim, we note the following:
    \begin{equation*}
        \int_{0}^{\infty} e^{-t\omega}(\psi_0*\psi_1^{*n})(t)\mass{t} = \left( \int_0^\infty e^{-t\omega}\psi_0(t)\mass{t} \right)\cdot \left( \int_0^\infty e^{-t\omega}\psi_1(t)\mass{t} \right)^n
    \end{equation*}
    Hence we can write $\int_0^\infty e^{-t\omega}\theta(t)\mass{t} =\left( \int_0^\infty e^{-t\omega}\psi_0(t)\mass{t} \right)\cdot (1-\int_0^\infty e^{-t\omega}\psi_1(t)\mass{t})^{-1} $, which is finite by our assumption.
\end{proof}
Now the proof of the theorem is achieved by combining all these estimates:
\begin{proof}[Proof of \Cref{thm: bounded perturbation}]
    We begin by proving that the Dyson-Phillips series converges uniformly in the strong operator topology for $t>0$ by bounding it using $\phi$ and $\psi$:
    \begin{equation*}
        \left\lVert \operatorname{Per}^n(e^{-tA}f) \right\rVert \leq \phi*\psi^{*n}(t) \quad \left\lVert B \operatorname{Per}^{n}(e^{-tA}f) \right\rVert \leq \psi^{*n} (t)
    \end{equation*}
    First note that the functions $\phi = \psi_0$ and $\psi = \psi_1$ satisfy the estimates in \Cref{lemma: Hille13.4.3}, by the assumption we made in \Cref{thm: bounded perturbation}. Also by our assumption on $e^{-tA}$, the case $n=0$ is trivially true. The case for general $n$ is shown by an easy induction with the following inequality:
    \begin{equation*}
        \left\lVert\operatorname{Per}^{k+1}(e^{-tA}) \right\rVert \leq \int_{0}^t \left\lVert e^{-(t-s)A} \right\rVert \cdot \left\lVert B\operatorname{Per}^{k}(e^{-sA}) \right\rVert \mass{s} \leq \phi*\psi^{*(k+1)}(t)
    \end{equation*}
    and similarly for $B\operatorname{Per}^{k+1}(e^{-tA})$:
    \begin{equation*}
        \left\lVert B\operatorname{Per}^{k+1}(e^{-tA}f) \right\rVert \leq \int_0^{t}\left\lVert Be^{-(t-s)A} \right\rVert \cdot \left\lVert B\operatorname{Per}^{k}(e^{-sA}) \right\rVert \mass{s} \leq \psi^{*(k+2)}(t)
    \end{equation*}
    Now $ \sum_{k=0}^\infty\left\lVert \operatorname{Per}^k(e^{-tA}) \right\rVert $ is bounded by $\theta(t)$, a fact we have proven in the above lemma to be uniformly convergent in $t$ on any interval $(\varepsilon, 1/\varepsilon)$ for $0<\varepsilon<1$. These show $e^{-t(A+B)}$ defines a strongly continuous heat semigroup.
\end{proof}

Fix now a finite-dimensional $G$-representation $(\tau, V_\tau)$, together with a compact subgroup $L\subseteq G$, such that $V_\tau|L$ is a unitary representation of $L$. Form now a $G$-invariant strongly elliptic operator $\mathcal{D}$ on the homogeneous vector bundle $E_\tau:= G\times_\tau V_\tau \to G/L$ like \cite[\S 1.1]{Barbasch1983}. Note the smooth sections $\Gamma(E)$ can be identified with the smooth vectors of the induced module $\ind_{L}^G (\tau)$. Explicitly,
\begin{equation}
    C^\infty(G; \tau) := \{F:G\to V_\tau \mid f\in C^\infty, F(gl)= \tau(l^{-1})F(g) \text{ for all }l\in L\},
\end{equation}
where $G$ acts on $C^\infty(G; \tau)$ via right regular representation. It is isomorphic with $\Gamma(E)$ via the canonical $G$-equivariant isomorphism
\begin{equation}\label{exp: bundle to module} 
    \mathcal{A}: \Gamma(E) \longrightarrow C^\infty(G; \tau) \qquad F(g):= (\mathcal{A}f)(g) = \tau(g^{-1})(f(gL)).
\end{equation}
Fix a complex basis $X_i$ of $ \mathfrak{g}_\complex$. Note that every $G$-invariant operator $\mathcal{D}$ can be identified via $\mathcal{A}$ with an operator of the following form
\begin{equation*}
    D_\tau = \sum_I X^I \otimes C_{I} \in [U( \mathfrak{g}_\complex )\otimes \End(V)]^L,
\end{equation*} 
where $L$ acts on $U( \mathfrak{g}_\complex )\otimes \End(V)$ by:
\begin{equation} \label{exp: action on differential operator}
    l\circ (\sum_i D_i\otimes C_i ) := \sum_i\Ad_l(D_i)\otimes (\tau(l)\circ C_i \circ \tau(l)^{-1}).
\end{equation}
We further assume that $\mathcal{D}$ is $G$-invariant, i.e., it commutes with the natural action of $G$ on the $E_\tau$. Then the heat kernel associated with the semigroup $e^{-tD_\tau}$ can be treated as a function $k_t^\tau: G\to \End(V)$ which satisfies the following covariance property: 
\begin{equation} \label{exp: L-equivariance}
    k^\tau_t(g) = \tau(a)k^\tau_t(a^{-1}gb)\tau(b)^{-1} \qquad \text{ for }x\in G, a,b \in L.
\end{equation}
We note that $D_\tau$ is often not an elliptic operator with constant coefficients because of the covariance in the direction of fiber of $E_\tau$. Nonetheless we can still estimate the growth via bounded perturbations on part of the operator with constant coefficients.

\begin{proposition}\label{prop: perturbed kernel estimates}
    Let $D_\tau$ be a $G$-invariant differential operator on $E_\tau$ as given above. Suppose $D_\tau = D + B$, with $D \in U( \mathfrak{g}_\complex )$ a $G$-invariant strongly elliptic operator with constant coefficients of order $m$; and $B$ a differential operator of order $\ell<m$ of the following form:
    \begin{equation} \label{exp: remainder operator}
        B = \sum_{|I|\leq \ell} D^I\otimes C_I \in U( \mathfrak{g}_\complex )\otimes \End(V).
    \end{equation} 
    Then the perturbed semigroup $e^{-t(\overline{D+B})}$ is strongly continuous in $t$, with $k_t^{D+B}$ its kernel. The derivatives of $k^{D+B}_t$ satisfy the following estimate: For each $\rho>0$ and every derivative $X^J$ of order $|J|<\ell$, 
    \begin{equation}
        \int_G e^{\rho|g|} \left\lVert  L(X^J) k_t^{D+B}(g)  \right\rVert _{\End(V)} \mass{g}  <\infty
    \end{equation}
    for all fixed $t>0$. 
\end{proposition}
\begin{proof}
    For clarity of exposition we abbreviate $L^p_{\tau}:= L^p(G; V_\tau)$ the $L^p$-spaces with coefficients in $V_\tau$ for $1\leq p\leq \infty$, and the weighted $L^p$-spaces in a similar way. Recall that $L^\infty_{\rho, \tau}(G)$ is the weighted $L^\infty$-space with norm 
    \begin{equation*}
        \left\lVert \varphi \right\rVert _{L^\infty_{\rho, \tau}} = \sup_{g\in G}e^{-\rho|g|}\left\lVert \varphi(g) \right\rVert _V.
    \end{equation*} 
    Denote $C_{\rho, \tau, \infty}^k = C^k(L_{\rho, \tau}^\infty(G))$ its $C^k$ vectors, and define its $C^k$-norms as \eqref{exp: Ck-norm}, replacing $H_\pi$ by $L_{\rho, \tau}^\infty(G)$.

    The strong continuity of the perturbed heat semigroup $e^{-t(\overline{D+B})}$ is straightforward from \Cref{thm: bounded perturbation}. Derivatives of the perturbed heat kernel can be estimated by bounding its weighted $L^\infty$-norm in a similar way as in the proof of \Cref{lemma: RobinsonIII.2.5}:
    \begin{equation}
        \begin{split}
            \int_G \left\lVert  L(X^J) k_t^{D+B}(g)  \right\rVert _{\End(V)} \mass{g}
            &= \sup_{\left\lVert \varphi \right\rVert _{L^\infty_{\rho, \tau}}\leq 1 }\left\{ \left\lVert L(X^J) e^{-tL^\rho(\overline{D+B})} \varphi(e_G) \right\rVert _V \right\} \\
            &\leq \sup_{\left\lVert \varphi \right\rVert _{L^\infty_{\rho, \tau}}\leq 1 }\left\{ \left\lVert e^{-tL^\rho(\overline{D+B})} \varphi \right\rVert _{C_{\rho, \tau, \infty}^{\ell}} \right\}.
        \end{split}
    \end{equation}
    As the Dyson-Phillips series \eqref{exp: Dyson-Phillips expansion} is absolutely convergent, it suffices to estimate the $C^k$-norm of each term $\operatorname{Per}^k(e^{-t\bar{D}}\varphi)$. Apply now \Cref{thm: bounded perturbation} further to the $C^{\ell}$-norms of each term by verifying the norm assumptions:
    \begin{itemize}[wide = 0pt]
        \item As $D$ acts on $L^\infty_{\rho, \tau}$ by $L^\rho\otimes \mathrm{Id}_V$,  
        \begin{equation*}
            \left\lVert e^{-tL^\rho(\bar{D})} \right\rVert _{L^\infty_{\rho, \tau} \to C^{\ell}_{\rho, \tau, \infty}}  = \left\lVert e^{-tL^\rho(\bar{D})} \right\rVert _{L^\infty_{\rho} \to C^{\ell}_{\rho, \infty}} \leq Me^{t\omega}
        \end{equation*}
        for some $M, \omega>0$, by the continuity bound \eqref{exp: continuity bound};
        \item For small $t\in (0,1]$ the small time estimate \eqref{exp: robinsonII.2.1} gives:
              \begin{equation}
                \begin{split}
                     \left\lVert Be^{-tL^\rho(\bar{D})} \right\rVert _{L^\infty_{\rho, \tau} \to L^\infty_{\rho, \tau}} &\leq \sum_I \left\lVert B_I \right\rVert _{\End(V)} \left\lVert D^Ie^{-tL^\rho(\bar{D})} \right\rVert _{L^\infty_{\rho} \to L^\infty_{\rho}} \\
                     & \leq a \sum_{I} \left\lVert e^{-tL^\rho(\bar{D})} \right\rVert _{L^\infty_{\rho} \to  C^\ell_{\rho, \infty}}\\
                     &\leq a' b^{\ell} \ell! t^{-\ell/m} \sim C_{\ell} t^{-\ell/m}
                \end{split}
              \end{equation}
              with $\ell/m <1$ by our assumption.
    \end{itemize}
    Now repeating the notations and arguments in \Cref{lemma: Hille13.4.3} we see that the $C^{\ell}_{\rho, \infty}$-norm is indeed finite. For $\varphi = \psi\otimes v\in L^2(G)\otimes V_\tau^* \cong L^2(G, V_\tau)$ of unit norm, 
    \begin{equation}
            \left\lVert e^{-tL^\rho(\overline{D+B})} \varphi \right\rVert _{C_{\rho, \tau, \infty}^{\ell}} \leq \sum_{k=0}^\infty \left\lVert \operatorname{Per}^k(e^{-t\bar{D}}\varphi) \right\rVert _{C_{\rho, \infty}^{\ell}} \leq \theta_{B, A}(t)
    \end{equation}
    with the absolutely convergent series $\theta_{B, A}(t)$, expressed in the sum of convolution products of functions in $t$, is finite:
    \begin{equation}
        \theta_{B, A}(t) = \sum_{k=0}^\infty  \left\lVert e^{-t\bar{D}} \right\rVert _{L^\infty_\rho \to C_{\rho, \infty}^{\ell}} * \left\lVert Be^{-t\bar{D}} \right\rVert _{L^\infty_{\rho, \tau} \to L^\infty_{\rho, \tau}}^k (t) <\infty.
    \end{equation}
    As \Cref{lemma: Hille13.4.3} proved, the series $\theta_{B, A}(t)$ is majorized by $\sum_{n}2^{-n}t^{-2}e^{-t\omega_1}$ for some positive $\omega_1>\omega>0$. This proves that $e^{\rho|g|} \left\lVert  \partial^I_g k_t^{D+B}  \right\rVert _{\End(V)}\in L^1(G)$ for all $t>0$.
\end{proof}

This estimate shows that $e^{\rho|g|}L_{X^I}k_t^{D+B}\in L^1(G, \End(V))$ for any $\rho>0$ and any differentiation $X^I\in U( \mathfrak{g}_\complex )$ from the left. To establish Schwartz estimate, we only need the following variant of the Sobolev lemma adapted to Lie groups \cite[Lemma~5.1]{Poulsen1972a}, which follows directly from the Sobolev lemma in $\real^n$ by choosing local coordinates:
\begin{lemma}\label{lemma: Sobolev lemma}
    Fix an integer $s>\dim G$. Then for each compact neighborhood $\mathcal{B}(e_G)$ of $e_G$ there exists a constant $C$ such that:
    \begin{equation}\label{exp: sobolev}
        f(e_G) \leq C\sum_{|I|\leq s}\int_{\mathcal{B}(e_G)} \left\lvert L_{X^I}f(y) \right\rvert \mass{y}
    \end{equation}
    for all $f\in C^\infty(G)$. 
\end{lemma}
We have now prepared for the proof of the main result \Cref{thm: Schwartz kernel}. Recall the definition of Schwartz spaces \eqref{exp: Schwartz space}. 
\begin{proof}[Proof of \Cref{thm: Schwartz kernel}]
    First we dominate the growth of $\phi_0^{G/Z}(x)$ by choosing appropriate $e^{\rho|g|}$. From \eqref{exp: spherical vector estimate} we see that:
\begin{equation}
    -\log \phi_0^{G/Z}(x) \leq \gamma d(e_G, x) \quad \text{ for all } x\in G
\end{equation}
for some constant $\gamma>0$. Hence if we choose the constant $\rho>\gamma>0$ in \Cref{prop: perturbed kernel estimates} this indeed implies $\phi_0^G(x)^{-2/p}L_{X^I}k_t^{D+B}\in L^1(G, \End(V))$ for all $t>0$ and $p>0$. Apply Sobolev inequality \eqref{exp: sobolev} now:
\begin{equation} \label{exp: MS163}
    \begin{split}
        &\left\lVert \phi_0^G(x)^{-2/p}L_{X^I}k_t^{D+B}(x) \right\rVert _{\End(V)} \\
        \leq& C\sum_{|J|\leq s}\int_B \left\lVert L_{X^I}\phi_0^G(yx)^{-2/p}L_{X^J}k_t^{D+B}(yx)\right\rVert _{\End(V)}\mass{y}\\
         \leq &C\sum_{|J|\leq s}\left\lVert L_{X^\alpha}\phi_0^G(yx)^{-2/p}L_{X^J}k_t^{D+B} \right\rVert _{L^1(G; \End(V))}<\infty.
    \end{split}
\end{equation}
So we have proved that $\nu_{D_1, -, r}(F)$ is finite for all $D_1\in U(\mathfrak{g}_\complex)$ and $r>0$. Next we establish a relation between $L_{X^I}R_{X^J}k_t$ and $L_{X^{I'}}k_t$. To prove this we use the fact that $k_t$ is $\Ad(K)$-invariant, i.e.: 
\begin{equation*}
    k_t(x^{-1}gx) = k_t(g)
\end{equation*} 
for all $x\in K$. This follows readily from the fact that $K$ acts on $V$ by unitary operators, and acts on $ \mathfrak{g}$ by isometries with respect to the metric $B^\theta$.


The following lemma is modelled after an argument of Harish-Chandra \cite[\S11, Lemma~17]{harish1966}:
\begin{lemma} \label{lemma: bounded two sides by one side}
    Let $V$ be a unitary $K$-bimodule and $f:G \to V $ be an $\Ad(K)$-invariant $C^\infty$-function. Then for each pair $(D, D')\in U( \mathfrak{g}_\complex)\times U( \mathfrak{g}_\complex )$, one can choose a finite number of $D_i\in U( \mathfrak{g}_\complex ) (1\leq i \leq p)$ such that:
    \begin{equation}
        \left\lVert L_{D}R_{D'}f(g) \right\rVert _V \leq \sum_{j=1}^p \left\lVert L_{D_j}f(g) \right\rVert  _V\quad \left\lVert L_{D}R_{D'}f(g)  \right\rVert _V; \leq \sum_{j=1}^p \left\lVert R_{D_j}f(g) \right\rVert _V.
    \end{equation}
\end{lemma}
\begin{proof}
Decompose the Lie algebra $\mathfrak{g = k\oplus a \oplus n_p}$ via the Iwasawa decomposition as in \eqref{exp: Iwasawa}. Recall $ \mathfrak{n_p} = \sum_{i=1}^l  \mathfrak{g}_{\alpha_i} $ is the space containing all the simple positive $( \mathfrak{g};  \mathfrak{a} )$-roots $\{\alpha_i\}_{i=1}^l$. Applying the Poincaré-Birkhoff-Witt theorem, we decompose the universal enveloping algebra as:
\begin{equation} \label{exp: PBW}
    U( \mathfrak{g}_\complex ) = U( \mathfrak{k}_\complex )  U( \mathfrak{a}_\complex )  U( \mathfrak{n}_\complex ).
\end{equation}
Here $U(\mathfrak{l}_\complex)$ are the corresponding universal enveloping algebra generated by corresponding algebra $\mathfrak{l}$.

Fix now an integer $d \geq 0$ such that $D, D' \in U_d( \mathfrak{g}_\complex )$. Denote $U_d( \mathfrak{g}_\complex )\subseteq U( \mathfrak{g}_\complex )$ the subset containing all elements of order $\leq d$. We can choose a basis $\{D_{ \mathfrak{k}} D_{ \mathfrak{a}} D_{ \mathfrak{n} }\}$ of $U_d( \mathfrak{g}_\complex )$ where $ D_{ \mathfrak{k} }\in U( \mathfrak{k}_\complex )$, $ D_{ \mathfrak{a} }\in U( \mathfrak{a}_\complex )$ and $ D_{ \mathfrak{n} }\in U( \mathfrak{n}_\complex )$. Denote this basis as $\mathcal{B}_d$. Moreover, recall that the natural action of $G$ on $U( \mathfrak{g}_\complex )$ extends the adjoint action on $\mathfrak{g}_\complex \subseteq U_1( \mathfrak{g}_\complex )$. Then:
\begin{equation}\label{exp: adj on nilpotent}
    \Ad(a)D_{ \mathfrak{n} } = \exp\left( \sum_{1 \leq i \leq l}m_i \alpha_i(\log a) \right)D_{ \mathfrak{n} }
\end{equation}
for all $a\in A = \exp \mathfrak{a}$. This equality is obtained by extending the map $\Ad(a)  \mathfrak{g} _{\alpha_i} = e^{\alpha(\log a)} \mathfrak{g} _{\alpha_i}$. Moreover, as $U(\mathfrak{n}_\complex)$ is generated by $g_{\alpha_i}$ as an algebra, we see that $m_i$ are non-negative integers. For the same reason we can expand the $\Ad(K)$-action based on its action on the basis:
\begin{equation}
    \Ad(k)D = \sum_{b\in \mathcal{B}_d} a_b( k )b \qquad \Ad(k)D' = \sum_{b\in \mathcal{B}_d} a'_b( k )b
\end{equation}
for $k\in K$.  Observe that $U_d(\mathfrak{g}_\complex)$ is a continuous representation of $K$, hence $a_b$ and $a_b'$ are continuous functions in $K$. 

Next by the $KAK$-decomposition \cite[Lemma~2.5]{HerbSchwartz4}, we write $G = KA^+K$ where $A^+ = \exp(\overline{ \mathfrak{a^+_p} })$ where $ \mathfrak{a}^+$ contains all $H\in  \mathfrak{a}$ such that $\alpha_i(H) \geq 0$ for all restricted roots $\alpha_i$. Denote:
\begin{equation}
    c = \sup_{k\in K} \max_{b\in B} (|a_b(k)|, |a'_b(k)|)
\end{equation}  
This constant is clearly finite if $K$ is compact. In the case $K$ is noncompact, we can appeal to \eqref{exp: refined Cartan decomposition} and then write $a_b$ and $a_b'$ as continuous functions that are $Z$-invariant instead \cite[Lemma~2.7]{HerbSchwartz4}. In all cases, $c$ is finite. Next we estimate growths in the $A$-direction under derivations. Write $g = k_1ak_2\in KA^+K$, then:
\begin{equation} \label{exp: growth in A direction}
    \|L_{D}R_{D'}f(g)\|_V \leq \left\lVert L_{(\Ad_{k_1^{-1}}D)}R_{(\Ad_{k_2}D')}f(a) \right\rVert _V \leq \sum_{b, b'\in \mathcal{B}_d} c^2\left\lVert L_bR_{b'}f(a) \right\rVert _V.
\end{equation}
Here the first inequality uses the fact that $V$ is a unitary $K$-bimodule
\begin{equation*}
    \|k_1vk_2\|_V = \|v\|_V
\end{equation*} 
for all $k_1, k_2\in K$ and $v\in V$. Now it suffices to estimate the norm for each $L_bR_{b'}f(a)$. Write each $b\in \mathcal{B}_d$ in the form of
\begin{equation*}
    b = D_{ \mathfrak{k} }^b D_{ \mathfrak{a} }^b D_{ \mathfrak{n} }^b
\end{equation*}
with respect to the decomposition \eqref{exp: PBW}. Using the fact $f$ is $\Ad(K)$-invariant, we see $L_{D_{ \mathfrak{k} }}f = R_{D_{ \mathfrak{k} }}f$ for any $D_{ \mathfrak{k} }\in U( \mathfrak{k}_\complex)$. Now combining the fact that $ \mathfrak{a}$ is abelian and the identity \eqref{exp: adj on nilpotent}, one shifts the actions on the left to the right side one by one:
\begin{equation}
    L_bR_{b'}f(a) = L_{D_{ \mathfrak{k} }^b D_{ \mathfrak{a} }^b D_{ \mathfrak{n} }^b}R_{b'}f(a) = R_{D^b_{ \mathfrak{k}} D^b_{ \mathfrak{a} }\Ad_{a^{-1}}(D_{ \mathfrak{n} }^{b})b'}f(a)
\end{equation}  
Where $R_{D_1D_2}f = R_{D_2}(R_{D_1}f)$. Next, 
\begin{equation} \label{exp: growth in one direction}
    \left\lVert R_{D^b_{ \mathfrak{k}} D^b_{ \mathfrak{a} }\Ad_{a^{-1}}(D_{ \mathfrak{n} }^{b})b'}f(a) \right\rVert _V  \leq  \left\lVert R_{D^b_{ \mathfrak{k}}D^b_{ \mathfrak{a} }D^b_{ \mathfrak{n} } b'  } f(a) \right\rVert _V
\end{equation}
by combining \eqref{exp: adj on nilpotent} and the fact $\alpha_i(\log a) \geq 0$ since $\alpha_i$ is a positive root. At last one notices that the following finite set
\begin{equation*}
    \{D^b_{ \mathfrak{k} }D^b_{ \mathfrak{a} }D^b_{ \mathfrak{n} }b' \mid b, b'\in \mathcal{B}_d\}
\end{equation*}
spans a finite-dimensional subspace in $U( \mathfrak{g}_\complex )$. We denote a basis of this subspace as $\{l_j:1\leq j\leq p\}$. We can choose a uniform bound $C'>0$ such that each element in the above set can be written as sums $\sum_{1\leq j \leq p} \gamma_j l_j$ with $|\gamma_j| \leq C'$. We can now estimate the derivatives from both sides by combining \eqref{exp: growth in A direction} and \eqref{exp: growth in one direction}:
\begin{equation*}
    \left\lVert L_DR_{D'}f(g) \right\rVert _V  \leq \sum_{b, b'\in \mathcal{B}_d} c^2 \left\lVert R_{D^b_{ \mathfrak{k}}D^b_{ \mathfrak{a} }D^b_{ \mathfrak{n} } b'  } f(a) \right\rVert _V \leq \sum_{1\leq j \leq p} |\mathcal{B}_d| C'c^2 \left\lVert R_{l_j}f(g) \right\rVert _V.
\end{equation*}
Note we have again exploited the fact that $f$ is $\Ad(K)$-invariant and the fact that $K$ acts on $V$ by isometry to deduce the last inequality. This concludes the proof. 
\end{proof}
    
By inspecting \Cref{thm: generalized_Kuga}, we see that the Hodge Laplacian $\Delta_p$ associated with the Riemannian metric induced by $B^\theta$ can be written in the following form, by taking $\tau = R$ the right regular representation:
    \begin{equation*}
        \Delta_p = R(\Delta_0)\otimes \mathrm{Id}_{\wedge^p  \mathfrak{g}^*_\complex } + \text{terms with lower order}
    \end{equation*}
    where $R(\Delta_0) = -R(\bog)$ is a strongly elliptic operator, and the remainder terms can be expressed in the form of \eqref{exp: remainder operator}. Now combining \Cref{lemma: bounded two sides by one side} with the Schwartz estimate from one side \eqref{exp: MS163}, we see that the kernel of $\Delta_p$ is indeed a Schwartz function.
\end{proof}

\section{Schwartz estimates on nilpotent groups} \label{section: nilpotent}
The above Schwartz estimates can be extended to cases beyond the Hodge Laplacian and reductive groups. In a subsequent paper we discuss how the spinor Laplacian associated with the same Riemannian metric can be proven to be Schwartz class in the sense of \eqref{exp: Schwartz space} following the same proof.


As an direct consequence of our estimate \eqref{exp: MS163}, we can also prove that the Hodge Laplacian $\Delta_p$ on a nilpotent Lie group $N$ are of Schwartz class. We briefly recall the definition of Schwartz functions on $N$ from \cite[\S A.2]{Corwin1990}, and then prove the result as a corollary that will be useful in \cite[TOCHANGE]{Thesis}.

Let $N$ be a simply connected real nilpotent Lie group. If $N$ is given polynomial coordinates $\gamma: \real^n\to N$, that is,
\begin{equation*}
    \log \circ \gamma: \real^n\to  \mathfrak{n}
\end{equation*}
defines a polynomial map with polynomial inverse. Though it is not dependent on the choice of basis, typically if we choose $\{X_1, \dots, X_n\}$ to be a weak Malcev basis \cite[Theorem~1.1.13]{Corwin1990}, then
\begin{equation*}
    \gamma(x_1, \dots, x_n) = \exp(x_1X_1)\cdots \exp(x_nX_n)
\end{equation*}
defines a polynomial coordinate. Now we can define $\mathcal{S}(N)$ as the pullback of the Schwartz functions on $\real^n$. More explicitly, a function $f$ on $N$ is Schwartz if it is finite under all seminorms:
\begin{equation*}
    \left\lVert x^\alpha D^\beta (f\circ \gamma)\right\rVert _{L^\infty(\real^n)} <\infty
\end{equation*}
for all $D^\beta = \partial_{x_1}^{\beta_1}\cdots \partial_{x_n}^{\beta_n}$ and $x^\alpha = x^{\alpha_1}\cdots x_{\alpha_n}$ a monomial. Note this is equivalent to defining it as the space of smooth functions $f\in C^\infty(N)$ such that
\begin{equation}
    \left\lVert p_\alpha L(X^\alpha)f \right\rVert _\infty <\infty
\end{equation} 
for all $p_\alpha\in \complex[G]$ polynomials on $G$, and $X_\alpha\in U( \mathfrak{g}_\complex )$ under any chosen basis of $ \mathfrak{g}$ \cite[Corollary~A.2.3]{Corwin1990}. Note it suffices to consider left derivatives in this case, as the Lie brackets of nilpotent Lie groups can be expressed by polynomial functions.

\begin{corollary}
    Let $N$ be a simply connected nilpotent Lie group, and fix on $N$ a left $N$-invariant Riemannian metric. Then the heat kernel $k_t^p$ of the Hodge Laplacian $\Delta_p$ on $N$ is a Schwartz function.
\end{corollary}
\begin{proof}
    Recall first that the content of \Cref{thm: rep_Laplacian} holds for any Lie group $G$. If we fix $V$ to be the $L^2(N)$ with right-regular action, then we see
    \begin{equation*}
        \Delta_p = -R(\sum_{i}X_i^2) + \sum_i R(X_i) \otimes C_i
    \end{equation*}
    with $C_i\in \End(\wedge^*\mathfrak{n}^*)$. In particular, $\Delta_p$ satisfies the assumptions of \Cref{prop: perturbed kernel estimates}. Hence the kernel of $\Delta_p$ is a Schwartz function. Consequently, \begin{equation*}
        \int_N e^{\rho|g|} \left\lVert  L(X^J) k_t^{p}(g)  \right\rVert _{\End(\wedge^* \mathfrak{n}^* )} \mass{g}  <\infty.
    \end{equation*}
    Because the polynomial grows slower than $e^{\rho|g|}$, we have
\begin{equation*}
    p_IL_{X^I}k_t^{p}\in L^1(N, \End(\wedge^p \mathfrak{n}^* )).
\end{equation*} 
In fact, $L_{X^J}p_I$ remains polynomial therefore bounded for any $|J|\geq 0$, therefore,
\begin{equation}
    L_{X^J}(p_I L_{X^I}k_t^{p})\in L^1(N, \End(\wedge^p \mathfrak{n}^* )).
\end{equation}
Now \Cref{lemma: Sobolev lemma} and the bound \eqref{exp: MS163} applies similarly to the nilpotent case. 
\begin{equation} 
    \begin{split}
       \left\lVert p_I L_{X^I}k_t^{D+B}(x) \right\rVert _{\End(V)} \leq C\sum_{|J|\leq s}\left\lVert L_{X^J}p_I L_{X^I}k_t^{D+B} \right\rVert _{L^1(G; \End(V))}<\infty,
    \end{split}
\end{equation}
and the claim is proven.
\end{proof}

\bibliographystyle{alpha}
\bibliography{../references} 

\begin{thebibliography}{KMB12}

\bibitem[BM83]{Barbasch1983}
Dan Barbasch and Henri Moscovici.
\newblock {L2-index and the Selberg trace formula}.
\newblock {\em J. Funct. Anal.}, 53(2):151--201, 1983.

\bibitem[BW00]{borel2013}
Armand Borel and Nolan~R. Wallach.
\newblock {\em {Continuous Cohomology, Discrete Subgroups, and Representations of Reductive Groups}}, volume~67 of {\em Mathematical Surveys and Monographs}.
\newblock American Mathematical Society, Providence, Rhode Island, 2000.

\bibitem[CG90]{Corwin1990}
Lawrence Corwin and Frederick~P. Greenleaf.
\newblock {\em {Representations of nilpotent Lie groups and their applications. Part I, Basic theory and examples}}, volume~18.
\newblock Cambridge University Press, 1990.

\bibitem[Han24]{Thesis}
Zhicheng Han.
\newblock {\em {The Spectra on Lie Groups and Its Application to twisted L2-Invariants}}.
\newblock PhD thesis, Georg-August-University G{\"{o}}ttingen, 2024.

\bibitem[HC75]{harish1975}
Harish-Chandra.
\newblock {Harmonic analysis on real reductive groups I the theory of the constant term}.
\newblock {\em J. Funct. Anal.}, 19(2):104--204, jun 1975.

\bibitem[HC76a]{Harish-Chandra1976II}
Harish-Chandra.
\newblock {Harmonic analysis on real reductive groups. II - Wave-packets in the Schwartz space}.
\newblock {\em Invent. Math.}, 36(1):1--55, dec 1976.

\bibitem[HC76b]{Harish-chandra1976III}
Harish-Chandra.
\newblock {Harmonic Analysis on Real Reductive Groups III. The Maass-Selberg Relations and the Plancherel Formula}.
\newblock {\em Ann. Math.}, 104(1):117, jul 1976.

\bibitem[HC84]{harish1966}
Harish-Chandra.
\newblock {Discrete Series for Semisimple Lie Groups. II}.
\newblock {\em Collect. Pap.}, 116(1):1643--1753, 1984.

\bibitem[Her92]{HerbSchwartz4}
Rebecca~A Herb.
\newblock {The {S}chwartz space of a general semisimple {L}ie group. {IV}. {E}lementary mixed wave packets}.
\newblock {\em Compos. Math.}, 84(2):115--209, 1992.

\bibitem[HP74]{HIllePhillips74}
Einar Hille and Ralph~S Phillips.
\newblock {\em {Functional analysis and semi-groups}}.
\newblock American Mathematical Society Colloquium Publications, Vol. XXXI. American Mathematical Society, Providence, R.I., 1974.

\bibitem[HW86]{HerbWolf86II}
Rebecca~A Herb and Joseph~A Wolf.
\newblock {Rapidly decreasing functions on general semisimple groups}.
\newblock {\em Compos. Math.}, 58(1):73--110, 1986.

\bibitem[HW90]{HerbSchwartz1}
Rebecca~A Herb and Joseph~A Wolf.
\newblock {The {S}chwartz space of a general semisimple {L}ie group. {I}. {W}ave packets of {E}isenstein integrals}.
\newblock {\em Adv. Math.}, 80(2):164--224, 1990.

\bibitem[KMB12]{KUZNETSOVA2012}
Yulia Kuznetsova and Carine Molitor-Braun.
\newblock {Harmonic analysis of weighted Lp-algebras}.
\newblock {\em Expo. Math.}, 30(2):124--153, 2012.

\bibitem[Kna86]{knapp2016}
Anthony~W. Knapp.
\newblock {\em {Representation Theory of Semisimple Groups}}.
\newblock Princeton Landmarks in Mathematics. Princeton University Press, Princeton, dec 1986.

\bibitem[Lan60]{langlandsthesis}
Robert~P Langlands.
\newblock {\em {Semi-groups and representations of Lie groups}}.
\newblock PhD thesis, Citeseer, 1960.

\bibitem[Nel59]{Nelson:1959a}
Edward Nelson.
\newblock {Analytic Vectors}.
\newblock {\em Ann. Math.}, 70(3):572, nov 1959.

\bibitem[Pou72]{Poulsen1972a}
Niels~Skovhus Poulsen.
\newblock {On C infty-vectors and intertwining bilinear forms for representations of Lie groups}.
\newblock {\em J. Funct. Anal.}, 9(1):87--120, 1972.

\bibitem[Rob91]{RobinsonBook1991}
Derek~W Robinson.
\newblock {\em {Elliptic operators and {L}ie groups}}.
\newblock Oxford Mathematical Monographs. The Clarendon Press, Oxford University Press, New York, 1991.

\bibitem[Wol74]{Wolf1974}
Joseph~A Wolf.
\newblock {\em {The action of a real semisimple Lie group on a complex flag manifold. II. Unitary representations on partially holomorphic cohomology spaces}}.
\newblock Number 138 in Memoirs of the American Mathematical Society, No. 138. American Mathematical Society, Providence, R.I., 1974.

\end{thebibliography}

\end{document}